\documentclass[smallextended]{svjour3} 
\smartqed  

\usepackage{graphicx}
\usepackage{amsmath} 
\usepackage{booktabs} 
\usepackage[figuresright]{rotating}
\usepackage[colorlinks,
             linkcolor=red,       
            anchorcolor=blue,  
            citecolor=blue,        %
             ]{hyperref}
\usepackage{latexsym}
\usepackage[ruled]{algorithm2e}
\usepackage{algorithmic}
\usepackage{adjustbox}
\usepackage{marvosym}

\usepackage[top=1.0in, bottom=1.0in, left=1.0in, right=1.0in]{geometry}

\journalname{JOTA}

\begin{document}

\title{An inexact proximal DC algorithm with sieving strategy for rank constrained least squares semidefinite programming 
}

\titlerunning{s-iPDCA for RCLSSDP}        


\author{Mingcai Ding\textsuperscript{1}        \and
        Xiaoliang Song\textsuperscript{1} \and Bo Yu\textsuperscript{1} 
}


\institute{Mingcai Ding \at 
                        \email{Dingmc$\_$dlut@mail.dlut.edu.cn}
              \and
                  Xiaoliang Song(\Letter)\\ 
              \email{Songxiaoliang@dlut.edu.cn}
              \at The research of this author was supported by "the Fundamental Research Funds for the Central Universities (DUT20RC(3)079)"
              \and
               Bo Yu\\\email{Yubo@dlut.edu.cn} 
               \at The research of this author was supported by "National Natural Science Foundation of China" (NO.11971092)
           \and
            \textsuperscript{1}Shcool of Mathematical Sciences, Dalian University of Technology, Dalian, LiaoNing, People’s Republic of China
}

\date{Received: date / Accepted: date}

\maketitle

\begin{abstract}
In this paper, the optimization problem of the supervised distance preserving projection (SDPP) for data dimension reduction (DR) is considered, which is equivalent to a rank constrained least squares  semidefinite programming (RCLSSDP). In order to overcome the difficulties caused by rank constraint, the difference-of-convex (DC) regularization strategy was employed, then the RCLSSDP is transferred into a series of least squares semidefinite programming with DC regularization (DCLSSDP). An inexact proximal DC algorithm with sieving strategy (s-iPDCA) is proposed for solving the DCLSSDP, whose subproblems are solved by the accelerated block coordinate descent (ABCD) method. Convergence analysis shows that the generated sequence of s-iPDCA globally converges to  stationary points of the corresponding DC problem. To show the efficiency of our proposed algorithm for solving the RCLSSDP,  the s-iPDCA is compared with classical proximal DC algorithm (PDCA) and the PDCA with extrapolation (PDCAe) by performing DR experiment on the COIL-20 database, the results show that the s-iPDCA  outperforms the PDCA and the PDCAe in solving efficiency. Moreover, DR experiments for face recognition on  the ORL database and the YaleB database demonstrate that the rank constrained kernel SDPP (RCKSDPP) is  effective and competitive by comparing the recognition accuracy with kernel semidefinite SDPP (KSSDPP) and kernal principal component analysis (KPCA).
\end{abstract}

\section{Introduction}\label{intro}
Dimensionality reduction (DR) plays a crucial role in handling high-dimensional data. The supervised distance preserving projection (SDPP) is proposed by  Zhu et al. \cite{Ref_Zhu2013Supervised} for data dimension reduction (DR), which minimizes the difference of local structure between projected input covariates and their corresponding responses. Given sample points $\{\boldsymbol{x}_1, \boldsymbol{x}_2, \cdots,\boldsymbol{x}_n\}\subset\mathcal{X}\subset\mathcal{R}^d$ and their corresponding responses $\{\boldsymbol{y}_1,\boldsymbol{y}_2, \cdots, \boldsymbol{y}_n\} \subset \mathcal{Y} \subset\mathcal{R}^m$, the form of the SDPP is as following:
\begin{equation}\label{eq_1}
\begin{array}{ll}
\min & J(\mathbf{P})=\frac{1}{n} \sum_{i=1}^{n} \sum_{\boldsymbol{x}_{j} \in \mathcal{N}\left(\boldsymbol{x}_{i}\right)}\left( \|\mathbf{P}^T\boldsymbol{x}_{i} - \mathbf{P}^T\boldsymbol{x}_{j}\|^2-\left\|\boldsymbol{y}_{i}-\boldsymbol{y}_{j}\right\|^{2}\right)^{2}. 
\end{array}
\end{equation}
In here, $\mathcal{N}\left(\boldsymbol{x}_{i}\right)$ is a neighborhood of $\boldsymbol{x}_{i}$. $\mathbf{P} \in \mathcal{R}^{d\times r} \left(d >> r\right)$ denotes the protection matrix. $\|\mathbf{P}^T\boldsymbol{x}_{i} - \mathbf{P}^T\boldsymbol{x}_{j}\|$ and $\left\|\boldsymbol{y}_{i}-\boldsymbol{y}_{j}\right\|$  are the pairwise distances among the projected input covariates and distances among responses, respectively.  Zhu et al. \cite{Ref_Zhu2013Supervised} reformulated the SDPP into a semidefinite quadratic linear programming (SQLP) when $d$ is small $(d\leq 100)$. Jahan \cite{Ref_jahan2018dimension} incorporated the total variance of projected covariates to the objective function of the SDPP, and transferred the obtained optimization problem into semidefinite least squares
SDPP (SLS-SDPP). In fact, the hidden low rank constraint has been ignored in the above two converted SDP problems, which may lead to suboptimal projection matrix and poor DR performance. In this manuscript, we show that the optimization problem of the SDPP \eqref{eq_1} can be converted equivalently into a rank constrained least squares semidefinite programming (RCLSSDP). The RCLSSDP also occurs in many other contexts such as nearest correlation matrix(NCM) \cite{Ref_gao2010structured,Ref_gao2010majorized,Ref_qi2014computing}, sensor network localization(SNL) \cite{Ref_singer2008remark} and classical multidimensional scaling (MDS) \cite{Ref_shang2004localization,Ref_1952Multidimensional}.

As we known, the rank constrained matrix optimization problems are computationally intractable \cite{Ref_buss1999computational} and NP hard because they are unconvex. The convex relaxation technique has been applied to remove the unconvexity of directly solving the rank constrained least squares problem (RCLS) \cite{Ref_candes2011tight,Ref_recht2010guaranteed}, one of the most popular convex relaxation is the nuclear norm minimization (NNM): 
\begin{equation}\label{eq_2}
\begin{array}{ll}
\min &\|\mathcal{A}\left(\mathbf{U}\right)- \boldsymbol{b}\|^{2}+\lambda\|\mathbf{U}\|_*\\
\text { s.t. } & \mathbf{U} \in \mathcal{R}^{p_1\times p_2}.
\end{array}
\end{equation}
where $\|\mathbf{U}\|_* = \Sigma_{i=1}^{\min{\left(p_1,p_1\right)}} \sigma_i(\mathbf{U})$ is the nuclear norm of $\mathbf{U}$ and $\lambda$  is a tuning parameter. Some efficient algorthms have been proposed to solve NNM, such as proximal gradient descent \cite{Ref_toh2010accelerated} and proximal point methods \cite{Ref_jiang2014partial}. It has been shown that the solution of the NNM has desirable properties under proper assumptions.  Another widely considered convex relaxation for the rank constrained optimization is called the max norm minimization (MNM) \cite{Ref_cai2014sparse,Ref_lee2010practical}, which use max norm as regularizer. However, these convex relaxation teachniques may expand the parameter space of target problem, which would keep the solution of convex relaxation problem away from that of target problem. What's more, the penalty parameter  need to be carefully tuned to ensure that the optimal solution of penalty problem satisfies low rank constraint and its related the fitting term is  small enough.

In addition to convex relaxation methods,  a class of non-convex optimization algorithms for solving rank constrained least squares has also been widely investgated \cite{Ref_chen2015fast,Ref_jain2013low,Ref_luo2020recursive}, which directly enforce $\operatorname{rank}(\mathbf{U})=r$ on the iterates. In these algorithms, the  low-rank matrix $\mathbf{U}$ is first factored to $\mathbf{R}\mathbf{L}^{\top}$ with two factor matrices $\mathbf{R}\in \mathcal{R}^{p_1\times r}$ and $\mathbf{L}\in \mathcal{R}^{p_2\times r}$, then alternately run either gradient decent or  minimization on $\mathbf{R}\in \mathcal{R}^{p_1\times r}$ and $\mathbf{L}\in \mathcal{R}^{p_2\times r}$ \cite{Ref_candes2015phase,Ref_li2019rapid,Ref_zheng2015convergent}. Based on this framework, some  methods perform sketching to speed up the computation via dimension reduction has been explored extensively in recent years \cite{Ref_luo2020recursive,Ref_mahoney2010randomized}. Since this class of method performs computation and storage the iterates in factored form, it is more efficient than the NNM both in computation  and storage, especially in cases where the rank $r$ is very small compared to $p_1$ and $p_2$. \par
The another popular unconvex method to directly deal with the rank constraint is that the rank constraint $\operatorname{rank}(\mathbf{U})\leq r$ is equivalently converted into the equality constraint $\|\mathbf{U}\|_*- \|\mathbf{U}\|_{(r)} = 0$ with  nuclear norm function $\|\mathbf{U}\|_*$ and Ky-Fan r-norm function $\|\mathbf{U}\|_{(r)}$. Then the exact penalty approach is used to penalize the euqality constraint into objective with chosen penalty parameter. Since the nuclear norm function $\|\mathbf{U}\|_*$ and the Ky-Fan r-norm function $\|\mathbf{U}\|_{(r)}$ are both convex, so the rank constrained optimization problem is converted into difference-of-convex (DC) probelm, which can be solved under the DC approach framework \cite{Ref_tao1997convex}. Then the rank constrained semidefinite programming (RCSDP) can  be reformulated as a DC problem \cite{Ref_gao2010structured,Ref_gao2010majorized}, so the classical DC algorithm is  used to solve the DC formulation of the RCSDP. In each iteration of the classical DC algorithm,  the concave part of the objective function is replaced by its  linearization at the chosen point, then  the resulting convex optimization problem is solved efficiently by some state-of-the-art solvers. In recent decades, massive algorithms for solving DC models have been proposed, including the classical DC algorithm \cite{Ref_an2005dc,Ref_le1999Exact,Ref_le2012exact,Ref_tao1997convex}, the proximal DC algorithm \cite{Ref_souza2016global}, the DC alorithm with proximal bundle \cite{Ref_gaudioso2018minimizing,Ref_de2019proximal} and the proximal DC algorithm with extrapolation \cite{Ref_liu2019refined,Ref_wen2018proximal}, etc.   In the literature \cite{Ref_jiang2021proximal}, a semi-proximal DC algorithm (SPDCA) was proposed to solve the rank constrained SDP and the large scale subproblem was solved by an efficient majorized semismooth Newton-CG augmented Lagrangian method based on the software package SDPNAL+ \cite{Ref_sun2020sdpnal+}. This technique performs very well in massive unconvex problems, e.g. quadratic assignment problem (QAP), standard quadratic programming and  minimum-cut graph tri-partitioning problem. However, they solve the subprobelm exactly, which will waste massive computation and reduce the solving efficiency.
\par
As all we known, the calculation of DC algorithm mainly focuses on solving convex subproblems. For large scale DC problem, we need huge amount of computation when solving convex subproblems to higher precision. In fact, it is time-consuming and unnecessary to solve the  convex subproblems exactly at each iterate of general DC algorithm, specially at the initial iteration of the DCA.   To overcome this issue, an inexact  proximal DC algorithm (iPDCA) is proposed \cite{Ref_souza2016global}, but the inexact strategy in their algorithm  is difficult to implement in practical application. In fact, it is a challenging task to design an inexact proximal DC algorithm that can guarantee the theoretical convergence and good numerical performance for solving large scale DCLSSDP. As far as we know, there is no research in this field.

The main contributions of this paper can be divided into the following  four parts. First of all, we propose a numerical efficient inexact proximal DC algorithm with sieving strategy (s-iPDCA) for solving the RCLSSDP, and the sequence generated by proposed algorithm globally converges to any stationary point of the corresponding DC problem. Secondly,  for the subprobelm of the s-iPDCA, we design an very effective accelerated block coordinate descent (ABCD) method to solve its strongly convex dual probelm.  Thirdly, an operable and numerically simple inexact strategy is used to solve subprobelm of s-iPDCA efficiently.  Finally, we compare our s-iPDCA with the classical PDCA and the PDCA with extrapolation (PDCAe) for solvig the RCLSSDP from DR experiment on the COIL-20 database, the results demonstrate that our proposed s-iPDCA outperforms the other two algorithms.  We also perform face recognition experiments on standard face recognition databases, the ORL and YaleB, the resluts indicate the new RCKSDPP model is very effective to reduce the dimension  of face images data with complex distribution.

Below are some common notations to be used in this paper. We use $\mathcal{S}^q$ to denote the linear subspace of all $q\times q$ real symmetric matrices and use $\mathcal{S}^q_+\backslash\mathcal{S}^q_-$ to denote the positive$\backslash$negative semidefinite  matrix cone. Denotes $\|\cdot\|_F$  the Frobenius norm of matrices. $\|\cdot\|$ is used to represent the $l_2$ norm of vectors and matrices. Let $\boldsymbol{e}_i$ be the $i^{th}$ standard unit vector. Given
an index set $\mathcal{L} \subset \left\{1, \cdots, q\right\}$, $|\mathcal{L}|$ denotes the size of $\mathcal{L}$. We denote the vector and square matrix of all ones by $\boldsymbol{1}_q$ and $\mathbf{E}_q$ respectively, and denote the identity matrix by $\mathbf{I}_q$. We use “$\operatorname{vec}(\cdot)$” to denote the vectorization of matrices. If $\boldsymbol{z} \in \mathcal{R}^p $, then Diag($\boldsymbol{z}$) is a $p \times p$ diagonal matrix with $\boldsymbol{z}$ on the main diagonal. We denote the Ky-Fan k-norm of matrix $\mathbf{U}$ as $\|\mathbf{U}\|_{(k)} = \Sigma_{i=1}^k \sigma_i(\mathbf{U})$, in which, $\sigma_i(\mathbf{U})$ is the $i^{th}$ largest singular value of $\mathbf{U}$. $\langle\mathbf{U},\mathbf{A}\rangle = \Sigma_{i,j=1}^q\mathbf{U}_{i j}\mathbf{A}_{i j}$ is used to denote the inner product between square matrix $\mathbf{U}$ and square matrix $\mathbf{A}$.  Let $\lambda_1\geq\lambda_2\geq\cdots\geq\lambda_q$ be the eigenvalues of $\mathbf{U}\in \mathcal{S}^q$ being arranged in nonincreasing order. We denote $\mu:=\left\{i|\lambda_i>0,i = 1,...,q\right\}$ and  $\nu:=\left\{i|\lambda_i<0,i = 1,...,q\right\}$ as the index set of positive eigenvalues and the index set of negative eigenvalues, respectively. The spectral decomposition of $\mathbf{U}$ is given as $\mathbf{U} = \mathbf{Q}\mathbf{\Lambda}\mathbf{Q}^{\top}$ with \[\mathbf{\Lambda} 
= \left[\begin{array}{ccc}
\mathbf{\Lambda_{\mu}} & 0 & 0\\
0 & 0 & 0\\
0 & 0 & \mathbf{\Lambda_{\nu}},
\end{array}\right]\]
The semidefinite positive$\backslash$negative matrix cone projection of $\mathbf{U}$ is represented as \[\operatorname{\Pi}_{\mathcal{S}_+^q}\left(\mathbf{U}\right)=\mathbf{Q}_{\mu}\mathbf{\Lambda}_{\mu}\mathbf{Q}_{\mu}^{\top}\backslash\operatorname{\Pi}_{\mathcal{S}_-^q}\left(\mathbf{U}\right)=\mathbf{Q}_{\nu}\mathbf{\Lambda}_{\nu}\mathbf{Q}_{\nu}^{\top}.\]

\section{Rank constrained supervised distance preserving projection}
\label{sec:2}
Let $\mathbf{U} = \mathbf{P}\mathbf{P}^T$, $\boldsymbol{\tau}_{i j} = \boldsymbol{x}_i - \boldsymbol{x}_j$ and $\iota_{i j} = \boldsymbol{y}_i - \boldsymbol{y}_j$, the SDPP \eqref{eq_1} can be formulated as
\begin{equation}\label{eq_3}
\begin{array}{ll}
\min_{\mathbf{U}\in \mathcal{S}_+^d} J(\mathbf{U})=\frac{1}{n} \sum_{i ,j=1}^n \mathbf{G}_{i j} \left(\left\langle\boldsymbol{\tau}_{i j} \boldsymbol{\tau}_{i j}^{\top},\mathbf{U}\right\rangle-\iota_{i j}^{2}\right)^{2},
\end{array}
\end{equation} 
where $\mathbf{G}$ denotes a graph matrix, whose element is shown as
\begin{equation}\nonumber
\mathbf{G}_{i j} = \left\{\begin{array}{l}
1\quad \text{ if } \boldsymbol{x}_j\in\mathcal{N}(\boldsymbol{x}_i) \\
0\quad \text{ otherwise }
\end{array}\right.
\end{equation} 
Let $p = \Sigma_{i,j=1}^n\mathbf{G}_{ i j }$ and rearrange the items in \eqref{eq_3} with $\mathbf{G}_{i j}>0$  into one cloumn, we  obtain a least squares semidefinite programming (LSSDP)
\begin{equation}\label{eq_4}
\begin{array}{l}
\min_{\mathbf{U} \in \mathcal{S}^d_+} J(\mathbf{U}) =\frac{1}{n} \|\mathcal{A}\left(\mathbf{U}\right)- \boldsymbol{b}\|^{2},
\end{array}
\end{equation}
in which, $\boldsymbol{b} = \left[\iota_{(1)}^2,\iota_{(2)}^2,...,\iota_{(p)}^2\right]^{\top}$, $\mathcal{A}:\mathcal{S}^d_+\rightarrow \mathcal{R}^p$ is a linear opterator that can be explicitly represented as
\begin{equation}\label{eq_5}
\mathcal{A}\left(\mathbf{U}\right) = \left[\langle \mathbf{A}_1,\mathbf{U}\rangle,\langle \mathbf{A}_2,\mathbf{U}\rangle,...,\langle \mathbf{A}_p,\mathbf{U}\rangle\right]^{\top}, \mathbf{A}_i = \tau_{(i)}\tau_{(i)}^{\top}, i = 1,2,...,p.
\end{equation}
   To distinguish the DR model  expressed in \eqref{eq_4} with the orginal SDPP\eqref{eq_1}, we call it semidefinite SDPP (SSDPP) when we use it to reduce the dimension of data. In fact, the LSSDP in \eqref{eq_4} is the convex relaxation of the optimization problem given in \eqref{eq_1}, but it may not satisfy the low rank constraint, i.e., $\operatorname{rank}\left(\mathbf{U}\right)\leq r,r\ll d$. Therefore, the low rank constraint should be added to the LSSDP in \eqref{eq_4}, given as
\begin{equation}\label{eq_6}
\begin{array}{ll}
\min_{\mathbf{U} \in \mathcal{S}^d_+} &J(\mathbf{U}) = \frac{1}{n}\|\mathcal{A}\left(\mathbf{U}\right)- \boldsymbol{b}\|^{2}\\
\text { s.t. }&\operatorname{rank}\left(\mathbf{U}\right)\leq r.
\end{array}
\end{equation}
Then, we obtain a rank constrained least squares semidefinite programming (RCLSSDP), which is equivalent to the optimization problem of the orginal SDPP \eqref{eq_1}. When apply this model to DR, we call it  rank constrained SDPP (RCSDPP).

As we known, it is difficult to solve the rank constrained optimization problem. By the observation that $\operatorname{rank}\left(\mathbf{U}\right)\leq r$ if and only if $\|\mathbf{U}\|_*-\|\mathbf{U}\|_{(r)} =  0$ and $\forall  \mathbf{U}\in \mathcal{S}^d_+$,  $\|\mathbf{U}\|_* = \langle \mathbf{U},\mathbf{I}\rangle$, the RCLSSDP can be reformulated as the equivalent form
 \begin{equation}\label{eq_7}
\begin{array}{ll}
\min_{\mathbf{U} \in \mathcal{S}^d_+} &J(\mathbf{U}) = \|\mathcal{A}\left(\mathbf{U}\right)- \boldsymbol{b}\|^{2}\\
\text { s.t. }&\langle \mathbf{U},\mathbf{I}\rangle-\|\mathbf{U}\|_{(r)} =  0.
\end{array}
\end{equation}
Even if the problem \eqref{eq_6} is converted to the above form, the difficulty caused by rank constraint is not eliminated. To address this problem, we employ an exact penalty approach.
By penalizing the equality constraint in \eqref{eq_7} into the objective function, we obtain a LSSDP with DC regularization term (DCLSSDP), shown as
 \begin{equation}\label{eq_8}
\min_{\mathbf{U} \in \mathcal{S}^d_+} J_c(\mathbf{U}) = \frac{1}{n}\|\mathcal{A}\left(\mathbf{U}\right)- \boldsymbol{b}\|^{2}+ c\left(\langle \mathbf{U},\mathbf{I}\rangle-\|\mathbf{U}\|_{(r)}\right).
\end{equation}
For the penalty problem \eqref{eq_8}, we have the following conclusion.
\begin{proposition}\label{prop_1}
Let $\mathbf{U}_c^*$ be a global optimal solution to the penalized problem \eqref{eq_8}. If $\operatorname{rank}\left(\mathbf{U}_c^*\right)\leq r$, then $\mathbf{U}_c^*$ is an optimal solution of \eqref{eq_7}.
\end{proposition}
\begin{proof}
For the details of proof, one can refer to  \cite[Proposition 3.1]{Ref_gao2010majorized}.
\end{proof}
\begin{proposition}\label{prop_2}
 Let $\hat{\mathbf{U}}^*$ be an optimal solution to the LSSDP \eqref{eq_4} and $\mathbf{U}_r$ a feasible solution to problem \eqref{eq_7}. Given an  $\epsilon>0$, and the $c$ is chosen such that $\left(J(\mathbf{U}_r)-J(\hat{\mathbf{U}}^*)\right)/c\leq \epsilon$. Then 
 \begin{equation}\label{eq_9}
\langle \mathbf{U}_c^*,\mathbf{I}\rangle-\|\mathbf{U}_c^*\|_{(r)}<\epsilon \quad \text{and}\quad J(\mathbf{U}_c^*)\leq \bar{J}- c\left(\right\langle \mathbf{U}_c^*,\mathbf{I}\rangle-\|\mathbf{U}_c^*\|_{(r)})\leq \bar{J}.
\end{equation}
where $\bar{J} = J(\mathbf{U}^*)$, $\mathbf{U}^*$ is a global optimal solution of RCLSSDP \eqref{eq_7}.
\end{proposition}
\begin{proof}
For the details of proof, one can refer to \cite[Proposition 3.2]{Ref_gao2010majorized}.
\end{proof}
 From the Proposition \ref{prop_2}, it is easy to see  that an $\epsilon$-optimal solution to the RCLSSDP \eqref{eq_7} in the sense of \eqref{eq_9} is guaranteed by solving the penalized problem \eqref{eq_8} with a chosen penalty parameter $c$. This provides the rationale to replace the rank constraint in problem \eqref{eq_7} by the penalty function $c \left(\langle \mathbf{U},\mathbf{I}\rangle-\|\mathbf{U}\|_{(r)}\right)$. Obviously, the penalty parameter $c$ is essential to produce low rank solution for the DCLSSDP \eqref{eq_8}. Thus, how to find the appropriate penalty parameter and efficiently solve the corresponding penalty problem is very important. Firstly, for choosing appropriate penalty parameter, we have the following conclusion.
\begin{proposition}\label{prop_3}
If exist a  $\bar{c}>0$, let $\mathbf{U}^*_{\bar{c}}\in \mathcal{S}^d_+$ be one of the global optimal solution of \eqref{eq_8} with penalty parameter $\bar{c}$ such that $\|\mathbf{U}^*_{\bar{c}}\|_*-\|\mathbf{U}^*_{\bar{c}}\|_{(r)} = 0$. Then for any $c>\bar{c}$, $\mathbf{U}^*_{\bar{c}}$ is one of the global optimal solution of \eqref{eq_8} with  penalty parameter $c$.
\end{proposition}
\begin{proof}
We can rewirite the \eqref{eq_8} as 
 \begin{equation}\label{eq_10}
\min_{\mathbf{U} \in \mathcal{S}^d_+} J(\mathbf{U}) + c(\| \mathbf{U}\|_*-\|\mathbf{U}\|_{(r)}).
\end{equation}
From $\|\mathbf{U}^*_{\bar{c}}\|_*-\|\mathbf{U}^*_{\bar{c}}\|_{(r)} = 0$ and $\|\mathbf{U}\|_*-\|\mathbf{U}\|_{(r)} \geq 0$, we have
\[J(\mathbf{U}) + c(\|\mathbf{U}\|_*-\|\mathbf{U}\|_{(r)})\geq J(\mathbf{U}) + c(\|\mathbf{U}^*_{\bar{c}}\|_*-\|\mathbf{U}^*_{\bar{c}}\|_{(r)}),\]
in here, $c>0$. Adding $\bar{c}(\|\mathbf{U}\|_*-\|\mathbf{U}\|_{(r)})$ to both side of  the above inequality, we get 
\[J(\mathbf{U}) + (c+\bar{c})(\|\mathbf{U}\|_*-\|\mathbf{U}\|_{(r)})\geq J(\mathbf{U}) + c(\|\mathbf{U}^*_{\bar{c}}\|_*-\|\mathbf{U}^*_{\bar{c}}\|_{(r)})+\bar{c}(\|\mathbf{U}\|_*-\|\mathbf{U}\|_{(r)})\]
Due to the optimality of $\mathbf{U}^*_{\bar{c}}$ for solving \eqref{eq_10}, we deduce that
\[J(\mathbf{U}) + (c+\bar{c})(\|\mathbf{U}\|_*-\|\mathbf{U}\|_{(r)})\geq J(\mathbf{U}^*_{\bar{c}}) + (c+\bar{c})(\|\mathbf{U}^*_{\bar{c}}\|_*-\|\mathbf{U}^*_{\bar{c}}\|_{(r)})\]
holds, which implies $\mathbf{U}^*_{\bar{c}}\|_{(r)}$ solves $\min_{\mathbf{U} \in \mathcal{S}^n_+} J(\mathbf{U}) + (c+\bar{c})(\| \mathbf{U}\|_*-\|\mathbf{U}\|_{(r)})$. This completes the proof by setting $c = c+\bar{c}$.
\end{proof}
\begin{remark}
If exist a  $\hat{c}>0$, let $\mathbf{U}^*_{\hat{c}}\in \mathcal{S}^d_+$ be one of the global optimal solution of \eqref{eq_8} with penalty parameter $\hat{c}$ such that $\operatorname{rank}(\mathbf{U}^*_{\hat{c}}) = \hat{r}\leq r$,  then for any  $c>\hat{c}$, $\operatorname{rank}(\mathbf{U}^*_{c})=\hat{r}$  holds, in which, $\mathbf{U}^*_{c}$ is one of the global optimal solution of \eqref{eq_8} with  penalty parameter $c$. Thus, based on the Proposition \ref{prop_3} and  the Proposition \ref{prop_1}, we use a strategy of gradually increasing the penalty parameter to obtain the appropriate penalty parameter. The details for adjusting the penalty parameter is given in the  Algorithm \ref{alg_penalty}.
\end{remark}
\begin{algorithm}[htb]
\caption{Framework of adjusting the penalty parameter for solving the RCLSSDP \eqref{eq_7}}
\label{alg_penalty}
\begin{algorithmic}
\STATE \textbf{Step 0}. Choose $c_0>0$, give $\rho>1$, set $i = 0$.
\STATE \textbf{Step 1}. Solve the  problem \eqref{eq_8} with penalty parameter $c_i$ by s-iPDCA \ref{alg_s_iPDCA}, shown as \[\mathbf{U}^*_{c_i}=\text{arg}\min J_{c_i}(\mathbf{U})+\delta_{\mathcal{S}_+^d}(\mathbf{U})\]

\STATE \textbf{Step 2}. If $\operatorname{rank}(\mathbf{U}^*_{c_i})\leq r$ holds, stop and return $\mathbf{U}^*_{c_i}$, otherwise,  $c_{i+1} = \rho c_i$, $i= i+1$, go to $\textbf{Step 0}$. 
\end{algorithmic}
\end{algorithm}
\section{Inexact proximal DC algorithm with sieving strategy for solving the DCLSSDP}\label{sec:3}
\subsection{Algorithm framework for solving the DCLSSDP}\label{sec:3_1}
 Obviously, for fixed penalty parameter $c$, the  penalized  problem \eqref{eq_8} can be formulated into the following standard DC form
 \begin{equation}\label{eq_11}
\min_{\mathbf{U} \in \mathcal{S}^d_+} J_c(\mathbf{U}) = \underbrace{\frac{1}{n}\|\mathcal{A}\left(\mathbf{U}\right)- \boldsymbol{b}\|^{2}+c\langle \mathbf{U},\mathbf{I}\rangle}_{f_1(\mathbf{U})}-\underbrace{c\|\mathbf{U}\|_{(r)}}_{f_2(\mathbf{U})}.
\end{equation}
Let's first briefly review the classical proximal DC algorithm (PDCA) for solving DC problem \eqref{eq_11}, the detail is shown in Algorithm \ref{alg_PDCA}. 
\begin{algorithm}[htb]
\caption{Proximal DC algorithm for solving the DCLSSDP \eqref{eq_11}}
\label{alg_PDCA}
\begin{algorithmic}
\STATE \textbf{Step 0}. Given $c>0$, initize $\mathbf{U}^{0} \in \mathcal{S}_+^d$, $\mathbf{W}^{0}\in \partial f_2(\mathbf{U}^0)$, tolerance error $\varepsilon\geq 0$, proximal parameter $\alpha>0$, $k = 0$.
\STATE \textbf{Step 1}. 
\begin{equation}\label{eq_12}
\mathbf{U}^{k+1} = \text{arg}\min f_1\left(\mathbf{U}\right)- \langle \mathbf{U}, \mathbf{W}^{k}\rangle +\delta_{\mathcal{S}^d_+}\left(\mathbf{U}\right) + \frac{\alpha}{2}\|\mathbf{U}-\mathbf{U}^{k}\|_F^2,
\end{equation}

\STATE \textbf{Step 2}. If $\|\mathbf{U}^{k+1}-\mathbf{U}^{k}\|_F \leq \varepsilon$, stop and return $\mathbf{U}^{k+1}$.
\STATE \textbf{Step 3}. Choose $\mathbf{W}^{k+1} \in\partial f_2(\mathbf{U}^{k+1})$, set $k\leftarrow k+1$, go to $\textbf{Step 1}$.
\end{algorithmic}
\end{algorithm}
 As we known, it is computation expensive and time-consuming to solve the subproblem \eqref{eq_12} exactly at each iteration of the PDCA \ref{alg_PDCA}. Therefore, we need to solve the strongly convex SDP \eqref{eq_12} inexactly. How to design an inexact proximal DC algorithm that can guarantee the theoretical convergence and good numerical performance for solving large scale DC problem remain an open question. Some researchers proposed to solve the subproblem inexactly, i.e., the inexact solution of \eqref{eq_12} satisfies following KKT condition:
\begin{equation} \label{eq_13}
\mathbf{0} \in \nabla f_1\left(\mathbf{U}^{k+1}\right) -\mathbf{W}^{k}+\partial\delta_{\mathcal{S}^d_+}\left(\mathbf{U}^{k+1}\right)+\alpha\left(\mathbf{U}^{k+1}-\mathbf{U}^{k}\right)-\mathbf{\Delta}^{k+1},
\end{equation}
where $\mathbf{\Delta}^{k+1}$ is the inexact term. Equivalently, $\mathbf{U}^{k+1}$  is the optimal solution of the following problem:
\begin{equation}\label{eq_14}
\min f_1\left(\mathbf{U}\right)- \langle \mathbf{U}, \mathbf{W}^{k}\rangle +\delta_{\mathcal{S}^d_+}\left(\mathbf{U}\right) + \frac{\alpha}{2}\|\mathbf{U}-\mathbf{U}^{k}\|_F^2-\langle \mathbf{\Delta}^{k+1}, \mathbf{U}\rangle,
\end{equation}
If  $\|\mathbf{\Delta}^{k+1}\|_F\leq \epsilon_{k+1}$, the trial point $\mathbf{U}^{k+1}$ is the  $\epsilon_{k+1}$-inexact solution of strongly convex subproblem \eqref{eq_12}.
 In traditional inexact algorithm \cite{Ref_jiang2012An,Ref_sun2015An}, the sequence $\{\epsilon_{k+1}\}$ is assumed to satisfy the condition that its summability $\Sigma_{k=0}^{\infty}\epsilon_{k+1}$ is convergent. Although this kind inexact strategy is simple, it cannot guarante the convergence of corresponding DC algorithm. To make sure that the algorithm is convergent, Wang et al. \cite{Ref_wang2019task} assumed the condition that $\|\mathbf{\Delta}^{k+1}\|_F < \frac{\alpha}{2}\|\mathbf{U}^{k+1}-\mathbf{U}^{k}\|_F$ holds. However, this condition may be unreachable in numerical experiments because the inexact term $\mathbf{\Delta}^{k+1}$ is related to $\mathbf{U}^{k+1}$ implicitly. Recently, Souza et al. \cite{Ref_souza2016global} proposed an inexact  proximal DC algorithm (iPDCA), in their algorithm, the inexact solution $\mathbf{U}^{k+1}$ is assumed to satisfy
 \begin{equation}\nonumber
f_1\left(\mathbf{U}^{k}\right)-f_1\left(\mathbf{U}^{k+1}\right)-\left\langle \mathbf{W}^{k}, \mathbf{U}^{k}-\mathbf{U}^{k+1}\right\rangle \geq \frac{(1-\sigma)\alpha}{2}\left\|\mathbf{U}^{k+1}-\mathbf{U}^{k}\right\|_F^{2}
\end{equation}
and 
\begin{equation}\nonumber
\left\|\mathbf{W}^{k+1}-\mathbf{W}^{k}\right\|_F \leq \theta\left\|\mathbf{U}^{k+1}-\mathbf{U}^{k}\right\|_F
\end{equation}
with $\mathbf{W}^{k+1}\in\partial f_2(\mathbf{U}^{k+1})$ and $\|\mathbf{\Delta}^{k+1}\|_F<\frac{\sigma\alpha}{2}\left\|\mathbf{U}^{k+1}-\mathbf{U}^{k}\right\|_F$. Obviously, these conditions are also difficult to implement in practical numerical experiments.\par
 Different from the above inexact strategies, in addition to $\lim_{k\rightarrow\infty}\epsilon_{k+1}=0$, no assumpution else is made for $\{\epsilon_{k+1}\}$ and $\{\mathbf{\Delta}_{k+1}\}$ in our algorithm. This bring that the sequence $\left\{J_c(\mathbf{U}^k)\right\}$ generated by the new iPDCA with $\{\epsilon_{k+1}\}$ and $\{\mathbf{\Delta}_{k+1}\}$  may not decrease. To address this issue, we employ a sieving strategy in our proposed iPDCA (s-iPDCA). Specifically, we choose a  stability center $\mathbf{U}^{k+1}$ at each iteration of s-iPDCA, then choose $\mathbf{W}^{k+1} \in\partial f_2(\mathbf{U}^{k+1})$ and the $\mathbf{U}^{k+1}$ is set as the proximal point for the next iteration. Once a new trial point $\mathbf{V}^{k+1}$ satisfying the $\epsilon_{k+1}$-inexact condition is computed,  the following rule is used to update $\mathbf{U}^{k+1}$: if
\begin{equation}\label{eq_15}
\|\mathbf{\Delta}^{k+1}\|_F < \frac{(1-\kappa)\alpha}{2}\|\mathbf{V}^{k+1}-\mathbf{U}^{k}\|_F,
\end{equation}
then we say a serious step is performed and set $\mathbf{U}^{k+1}:= \mathbf{V}^{k+1}$; otherwise, we say a null step is performed and set $\mathbf{U}^{k+1} := \mathbf{U}^{k}$, which means that the stability center remains unchanged. In \eqref{eq_15}, $\kappa\in \left(0, 1\right)$ is a tuning parameter to balance the  efficiency of the s-iPDCA and the inexactness of solution for \eqref{eq_14}. It should be noted that, as shown in s-iPDCA \ref{alg_s_iPDCA}, when the test \eqref{eq_15} doesn't hold true, only the iteration counter $k$ and inexact error bound $\epsilon_{k+1}$ are changed.
\par
\begin{algorithm}[htb]
\caption{Inexact proximal DC algorithm with sieving strategy (s-iPDCA) for solving the DCLSSDP \eqref{eq_11}}
\label{alg_s_iPDCA}
\begin{algorithmic}
\STATE \textbf{Step 0}. Given $c>0$, initize $\mathbf{U}^{0} = \mathbf{V}^{0}\in \mathcal{S}_+^d$, $\mathbf{W}^{0}\in \partial f_2(\mathbf{U}^0)$, tolerance error $\varepsilon\geq 0$, non-negative monotone descent sequence $\{\epsilon_{k+1}\}$, proximal parameter $\alpha>0$, $k = 0$.
\STATE \textbf{Step 1}. Inexactly solve \eqref{eq_12} so that $\|\mathbf{\Delta}^{k+1}\|_F\leq \epsilon_{k+1}$ holds true, shown as
\begin{equation}\label{eq_16}
\mathbf{V}^{k+1} = \text{arg}\min f_1\left(\mathbf{U}\right)- \langle \mathbf{U}, \mathbf{W}^{k}\rangle +\delta_{\mathcal{S}^d_+}\left(\mathbf{U}\right) + \frac{\alpha}{2}\|\mathbf{U}-\mathbf{U}^{k}\|_F^2-\langle \mathbf{\Delta}^{k+1}, \mathbf{U}\rangle,
\end{equation}

\STATE \textbf{Step 2}. If $\|\mathbf{V}^{k+1}-\mathbf{U}^{k}\|_F \leq \varepsilon$, stop and return $\mathbf{V}^{k+1}$. 
\STATE \textbf{Step 3}. If \eqref{eq_15} holds true, set $\mathbf{U}^{k+1} := \mathbf{V}^{k+1}$, choose $\mathbf{W}^{k+1} \in\partial f_2(\mathbf{U}^{k+1})$, otherwise, set $\mathbf{U}^{k+1} := \mathbf{U}^{k}$, $\mathbf{W}^{k+1} :=\mathbf{W}^{k}$. Set $k\leftarrow k+1$, go to $\textbf{Step 1}$.
\end{algorithmic}
\end{algorithm}
In this paper, we choose the  $\mathbf{W}^{k+1} \in\partial f_2(\mathbf{U}^{k+1})$ as the following: let $\lambda_1\geq\lambda_2\geq\cdots\geq\lambda_d$ is the eigenvalues of $\mathbf{U}^{k+1}$ being arranged in
nonincreasing order, then $\mathbf{U}^{k+1}$ has spectral decomposition $\mathbf{U}^{k+1} = \mathbf{Q}\mathbf{\Lambda}\mathbf{Q}^{\top}$, where $\mathbf{\Lambda}$ is the diagonal matrix whose $i^{th}$ diagonal entry is $\lambda_i$,  and the $i^{th}$ column of $\mathbf{Q}$ is the eigenvector of $\mathbf{U}^{k+1}$ corresponding to the eigenvalue $\lambda_i$, given as $\boldsymbol{q}_i$. Then set $\mathbf{W}^{k+1}=c\Sigma_{i = 1}^r \boldsymbol{q}_i\boldsymbol{q}_i^{\top}$.\par

Next, we consider to solve the subproblem \eqref{eq_16} from its dual problem. If we ignore the inexact term $\mathbf{\Delta}^{k+1}$ of the subproblem \eqref{eq_16}, then, its dual problem can be equivalently formulated as the following minimization problem: 
\begin{equation} \label{eq_17_1}
\min \frac{n}{4} \|\boldsymbol{z}\|^{2}  + \langle \boldsymbol{z},\boldsymbol{b}\rangle+\delta_{\mathcal{S}^d_+}\left(\mathbf{Y}\right)+\frac{1}{2\alpha}\|\mathcal{A}^*\left(\boldsymbol{z}\right)-\mathbf{Y}-\mathbf{\Phi}^{k}\|_F^2,
\end{equation}
where $\mathbf{\Phi}^{k} = \mathbf{W}^{k}-c\mathbf{I} + \alpha \mathbf{U}^{k}$. 
The KKT condition for solving \eqref{eq_17_1} is given as
\begin{align}
\mathbf{\Pi}_{\mathcal{S}^d_+}\left(\mathbf{Y}+\frac{1}{\alpha}\left(\mathcal{A}^*\left(\boldsymbol{z}\right)-\mathbf{Y}-\mathbf{\Phi}^{k}\right)\right)&=\mathbf{Y}\label{eq_18}\\
\frac{n}{2}\boldsymbol{z} + \boldsymbol{b}+\frac{1}{\alpha}\mathcal{A}\left(\mathcal{A}^*\left(\boldsymbol{z}\right)-\mathbf{Y}-\mathbf{\Phi}^{k}\right)&= \boldsymbol{0} \label{eq_19}
\end{align}
This problem belongs to a general class of unconstrained, multi-block convex optimization problems with coupled objective function. Thus, this problem can be solved efficiently by  two-block accelerate block coordinate descent (ABCD) method. In order to solve the convex problem \eqref{eq_17_1} inexactly, we employ an efficient ABCD algorithm, shown in Algorithm \ref{alg_abcd}. \par
\begin{algorithm}[htb]
\caption{Acclerated block coordinate desent algorithm (ABCD)}
\label{alg_abcd}
\begin{algorithmic}
\STATE \textbf{Step 0}. Initize $\tilde{\mathbf{Y}}^{1}\in \mathcal{S}_+^d$, $\boldsymbol{z}^0$, inexact error bound $\zeta_k$ given in \eqref{eq_34_1}, accelation factor $t_1 =1$, $j\leftarrow 1$.
\STATE \textbf{Step 1}. Update $\boldsymbol{z}$:
\begin{equation}\label{eq_20}
\boldsymbol{z}^{j}= \text {arg}\min\frac{n}{4} \|\boldsymbol{z}\|^{2} + \langle \boldsymbol{z},\boldsymbol{b}\rangle+\frac{1}{2\alpha}\|\mathcal{A}^*\left(\boldsymbol{z}\right)-\tilde{\mathbf{Y}}^j-\mathbf{\Phi}^{k}\|_F^2 
\end{equation}
\STATE \textbf{Step 2}. Update $\mathbf{Y}$:
\begin{equation}\label{eq_21}
\mathbf{Y}^{j} = \text {arg}\min\delta_{\mathcal{S}^m_+}\left(\mathbf{Y}\right)+\frac{1}{2\alpha}\|\mathcal{A}^*\left(\boldsymbol{z}^j\right)-\mathbf{Y}-\mathbf{\Phi}^{k}\|_F^2
\end{equation}
\STATE \textbf{Step 3}. If $\left\|\frac{n}{2}\boldsymbol{z}^j + \boldsymbol{b}+\frac{1}{\alpha}\mathcal{A}\left(\mathcal{A}^*\left(\boldsymbol{z}^j\right)-\mathbf{Y}^j-\mathbf{\Phi}^{k}\right)\right\| \leq \zeta_k$, stop.\par 
\STATE \textbf{Step 4}. Compute $t_{j+1} = \frac{1+\sqrt{1+4t_j^2}}{2}$, $\beta_j = \frac{t_j-1}{t_{j+1}}$, $\tilde{\mathbf{Y}}^{j+1}=\mathbf{Y}^j+\beta_j\left(\mathbf{Y}^j-\mathbf{Y}^{j-1}\right)$; \par set $j\leftarrow j+1$, go to $\textbf{Step 1}$.
\end{algorithmic}
\end{algorithm}
In the Algorithm \ref{alg_abcd},  for the $\boldsymbol{z}$-subproblem \eqref{eq_20}, its solution is obtained by solving the following linear system:
\begin{equation} \label{eq_22_1}
\frac{n}{2}\boldsymbol{z} + \boldsymbol{b}+\frac{1}{\alpha}\mathcal{A}\left(\mathcal{A}^*\left(\boldsymbol{z}\right)-\tilde{\mathbf{Y}}^j-\mathbf{\Phi}^{k}\right)= \boldsymbol{0}.
\end{equation}
By using preconditioned conjugate gradient (PCG) method, we can solve the above linear system efficiently, especially when its scale is very large.
 For the $\mathbf{Y}$-subproblem, fortunately,  it has  closed form solution, given as
\begin{equation}\label{eq_23_1}
\begin{array}{ll}
\mathbf{Y}^{j} &= \text{arg}\min\delta_{\mathcal{S}^d_+}\left(\mathbf{Y}\right)+\frac{1}{2\alpha}\|\mathcal{A}^*\left(\boldsymbol{z}^j\right)-\mathbf{Y}-\mathbf{\Phi}^{k}\|_F^2\\
&= \operatorname{prox}_{\mathcal{S}_+^d}\left(\mathcal{A}^*\left(\boldsymbol{z}^j\right)-\mathbf{\Phi}^{k}\right)= \mathbf{\Pi}_{\mathcal{S}_+^d}\left(\mathcal{A}^*\left(\boldsymbol{z}^j\right)-\mathbf{\Phi}^{k}\right).
\end{array}
\end{equation}
Then, from the relation between primal variable and dual variable, we can obtain  a feasible solution $\mathbf{U}^{(j)}=\frac{\mathbf{Y}^j+\mathbf{\Phi}^{k}-\mathcal{A}^*\left(\boldsymbol{z}^j\right)}{\alpha}$  of \eqref{eq_16} at each iteration of ABCD \ref{alg_abcd}.
As we known, the residuals of KKT equation is a good choice for the termination condition of the Alrorithm \ref{alg_abcd}. 
From \eqref{eq_23_1}, we have
\begin{equation}\label{eq_24_1}
\begin{array}{lll}
&\mathbf{\Pi}_{\mathcal{S}^d_+}\left(\mathbf{Y}^j+\frac{1}{\alpha}\left(\mathcal{A}^*\left(\boldsymbol{z}^j\right)-\mathbf{Y}^j-\mathbf{\Phi}^{k}\right)\right)\\
&= \mathbf{\Pi}_{\mathcal{S}^d_+}\left(\mathbf{\Pi}_{\mathcal{S}_+^d}\left(\mathcal{A}^*\left(\boldsymbol{z}^j\right)-\mathbf{\Phi}^{k}\right)+\frac{1}{\alpha}\mathbf{\Pi}_{\mathcal{S}_-^d}\left(\mathcal{A}^*\left(\boldsymbol{z}^j\right)-\mathbf{\Phi}^{k}\right)\right) = \mathbf{Y}^j.
\end{array}
\end{equation} 
Thus, 
\begin{equation}\label{eq_25_1}
\mathbf{Y}^j = \mathbf{\Pi}_{\mathcal{S}^d_+}\left(\mathbf{Y}^j+\frac{1}{\alpha}\left(\mathcal{A}^*\left(\boldsymbol{z}^j\right)-\mathbf{Y}^j-\mathbf{\Phi}^{k}\right)\right)
\end{equation} 
holds at each iteration of Algorithm \ref{alg_abcd}. This means that the KKT condition \eqref{eq_18} exactly holds at each iteration of ABCD \ref{alg_abcd}. Based on this observation, we do not need to check  the equation \eqref{eq_18} with positive semidefinite cone projection any more, which save massive computation when the matrix dimension $d$ is large ($d>500$). Let \[\gamma^j :=\frac{n}{2}\boldsymbol{z^j} + \boldsymbol{b}+\frac{1}{\alpha}\mathcal{A}\left(\mathcal{A}^*\left(\boldsymbol{z}^j\right)-\mathbf{Y}^j-\mathbf{\Phi}^{k}\right),\] then the Algorithm \ref{alg_abcd} stops if $\|\gamma^j\|\leq\zeta_k$ hold.
\subsection{Algorithm details and low rank structure utilization}\label{sec:3_2}

Although the linear operator $\mathcal{A}$ can be transformed into the form of matrix-vector product, i.e., $\mathcal{A}(\mathbf{U}) = \mathbf{A}\operatorname{vec}(\mathbf{U})$, it is impractical to store the matrix $\mathbf{A}$ with size of $p\times d^2$ when the size of problem \eqref{eq_11} is large. Thus, we need to form the linear operator $\mathcal{A}$ and its conjugate $\mathcal{A}^*$ at each iteration of the ABCD \ref{alg_abcd} and the s-iPDCA \ref{alg_s_iPDCA}. Therefore, the major computation for the s-iPDCA \ref{alg_s_iPDCA} focuses on the positive semidefinite cone projection for solving $\mathbf{Y^j}$ in the ABCD \ref{alg_abcd} and the formation of operators $\mathcal{A}$ and its conjugate $\mathcal{A}^*$.\par
 We further reduce the computation and the storage of the ABCD \ref{alg_abcd} and the s-iPDCA \ref{alg_s_iPDCA} for solving the DCLSSDP \eqref{eq_11} by employing the low rank structure of solution $\mathbf{U}^k$. Let $\lambda_1\geq\lambda_2\geq\cdots\geq\lambda_d$ be the eigenvalues of $\left(\mathcal{A}^*\left(\boldsymbol{z}^j\right)-\mathbf{\Phi}^{k}\right)$ being arranged in nonincreasing order. Denote $\mu:=\left\{i|\lambda_i>0,i = 1,...,d\right\}$ and  $\nu:=\left\{i|\lambda_i<0,i = 1,...,d\right\}$. Then $\left(\mathcal{A}^*\left(\boldsymbol{z}^j\right)-\mathbf{\Phi}^{k}\right)$ has the following spectral decomposition
\begin{equation}\nonumber
\mathcal{A}^*\left(\boldsymbol{z}^j\right)-\mathbf{\Phi}^{k} = \mathbf{Q}\mathbf{\Lambda}\mathbf{Q}^{\top},\quad \mathbf{\Lambda} 
= \left[\begin{array}{ccc}
\mathbf{\Lambda_{\mu}} & 0 & 0\\
0 & 0 & 0\\
0 & 0 & \mathbf{\Lambda_{\nu}}
\end{array}\right]
\end{equation}
In here, $\mathbf{\Lambda}$ is the diagonal matrix whose $i^{th}$ diagonal entry is $\lambda_i$,  and the $i^{th}$ column of $\mathbf{Q}$ is the eigenvector of $\mathbf{U}^{k+1}$ corresponding to the eigenvalue $\lambda_i$. Thus, we have 
\begin{equation}\nonumber
\mathbf{Y}^j = \operatorname{\Pi}_{\mathcal{S}_+^d}\left(\mathcal{A}^*\left(\boldsymbol{z}^j\right)-\mathbf{\Phi}^{k}\right)=\mathbf{Q}_{\mu}\mathbf{\Lambda}_{\mu}\mathbf{Q}_{\mu}^{\top}
\end{equation}
From the relation between primal variable and dual variable, i.e., $\mathbf{U}^{(j)}=\frac{\mathbf{Y}^j+\mathbf{\Phi}^{k}-\mathcal{A}^*\left(\boldsymbol{z}^j\right)}{\alpha}$,  we have
\begin{equation}\nonumber
\mathbf{U}^{(j)} = -\frac{1}{\alpha}\left(\mathcal{A}^*\left(\boldsymbol{z}^j\right)-\mathbf{\Phi}^{k}-\mathbf{Y}^j\right)=-\frac{1}{\alpha}\mathbf{Q}_{\nu}\mathbf{\Lambda}_{\nu}\mathbf{Q}_{\nu}^{\top}.
\end{equation}
From the formulation of $\mathbf{Y}^j$ and $\mathbf{U}^{(j)}$, we know 
\[\operatorname{rank}(\mathbf{Y}^j)+\operatorname{rank}(\mathbf{U}^{(j)})= |\mu|+|\nu|\leq d.\]Then if $|\mu|< \frac{d}{2}$ holds true, we update $\mathbf{Y}^j$  by $\mathbf{Y}^j = \mathbf{Q}_{\mu}\mathbf{\Lambda}_{\mu}\mathbf{Q}_{\mu}^{\top}$, otherwise by $\mathbf{Y}^j = \mathcal{A}^*\left(\boldsymbol{z}^j\right)-\mathbf{\Phi}^{k}- \mathbf{Q}_{\nu}\mathbf{\Lambda}_{\nu}\mathbf{Q}_{\nu}^{\top}$.  In fact, only the second case that $|\nu|< \frac{d}{2}$ will occur in our algorithm \ref{alg_abcd}, so we can store the $\mathbf{Q}_{\nu}\sqrt{-\mathbf{\Lambda}_{\nu}}$ instead of $\mathbf{Q}_{\mu}\sqrt{\mathbf{\Lambda}_{\mu}}$ or $\mathbf{Q}_{\mu}\mathbf{\Lambda}_{\mu}\mathbf{Q}_{\mu}^{\top}$ to reduce storage cost. Let $\mathbf{V} := \mathbf{Q}_{\nu}\sqrt{-\mathbf{\Lambda}_{\nu}}\in\mathcal{R}^{n\times |\nu|}$. Then, we can formulate the operator $\mathcal{A\left(\mathbf{V}\mathbf{V}^{\top}\right)}$ as 
\begin{equation}\nonumber
\mathcal{A}\left(\mathbf{V}\mathbf{V}^{\top}\right)  = \left[\langle \mathbf{A}_1,\mathbf{V}\mathbf{V}^{\top}\rangle,\langle \mathbf{A}_2, \mathbf{V}\mathbf{V}^{\top}\rangle,...,\langle \mathbf{A}_p, \mathbf{V}\mathbf{V}^{\top}\rangle\right]^{\top},
\end{equation}
in here, $\langle \mathbf{A}_i,\mathbf{V}\mathbf{V}^{\top}\rangle$ can be computed as 
\begin{equation}\nonumber
\begin{array}{ll}
\langle \mathbf{A}_i,\mathbf{V}\mathbf{V}^{\top}\rangle  = \langle \tau_{(i)}\tau_{(i)}^{\top},\mathbf{V}\mathbf{V}^{\top}\rangle =\tau_{(i)}^{\top}\mathbf{V}\mathbf{V}^{\top}\tau_{(i)}=\left(\mathbf{V}^{\top}\tau_{(i)}\right)^{\top}\left(\mathbf{V}^{\top}\tau_{(i)}\right).
\end{array}
\end{equation}
Then, the amount of computation for formulating the operator $\mathcal{A}$ on rank-$|\nu|$ matrix is reduced from $O\left(p(d^2+d)\right)$ to $O\left((|\nu| d+|\nu|)p\right)$. This strategy will significantly reduce the computation when $|\nu|\ll d$. \par
Based on the above analysis, we form the linear operator $\mathcal{A}(\mathbf{Y}^j)$ as
\begin{equation}\nonumber
\begin{array}{ll}
\mathcal{A}\left(\mathbf{Y}^j\right)  &= \mathcal{A}\left(\mathcal{A}^*\left(\boldsymbol{z}^j\right)-\mathbf{\Phi}^{k}- \mathbf{Q}_{\nu}\mathbf{\Lambda}_{\nu}\mathbf{Q}_{\nu}^{\top}\right)\\
&=\mathcal{A}\mathcal{A}^*\left(\boldsymbol{z}^j\right)-\mathcal{A}\left(\mathbf{\Phi}^{k}\right)- \mathcal{A}\left(\mathbf{Q}_{\nu}\mathbf{\Lambda}_{\nu}\mathbf{Q}_{\nu}^{\top}\right)
\end{array}
\end{equation}
In here, the term $\mathcal{A}\mathcal{A}^*\left(\boldsymbol{z}^j\right)$ can be computed by matrix-vector product when the  matrix  $\mathbf{A}\mathbf{A}^{\top}\in \mathcal{S}_+^p$ is stored. For the term $\mathcal{A}\left(\mathbf{\Phi}^{k}\right)$, we compute it as 
\[\mathcal{A}\left(\mathbf{\Phi}^{k}\right)= \mathcal{A}\left(\mathbf{W}^{k}\right)-\mathcal{A}\left(c\mathbf{I}\right) + \alpha \mathcal{A}\left(\mathbf{U}^{k}\right)\]
 from the definite of $\mathbf{\Phi}^{k}$. We note that $\mathbf{W}^{k}$ and $\mathbf{U}^{k}$ are also low rank matrices, and this matrices can be reformulated into factored form easily. In addition, for the term $\frac{1}{\alpha}\mathcal{A}\left(\mathcal{A}^*\left(\boldsymbol{z}^j\right)-\mathbf{Y}^j-\mathbf{\Phi}^{k}\right)$ in the KKT equation \eqref{eq_19}, it can be reformulated as  \[\mathcal{A}\left(\mathcal{A}^*\left(\boldsymbol{z}^j\right)-\mathbf{Y}^j-\mathbf{\Phi}^{k}\right) =  \mathcal{A}\left(\mathbf{Q}_{\nu}\mathbf{\Lambda}_{\nu}\mathbf{Q}_{\nu}^{\top}\right),\] then the KKT equation  \eqref{eq_19} can be checked by using the tricks mentioned above.\par

\subsection{A reliable inexact strategy} \label{sec:3_3}
To get the inexact error bound $\zeta_k$ for the ABCD algorithm \ref{alg_abcd}, we  need to find the relation between the inexact term $\mathbf{\Delta}^{k+1}$ in the primal problem \eqref{eq_16} and the inexact term $\gamma^j$ in minimization problem \eqref{eq_17_1}.
Firstly, we show how to obtain the inexact term $\mathbf{\Delta}^{k+1}$ and the corresponding inexact solution $\mathbf{V}^{k+1}$ from KKT equation of equation \eqref{eq_16}. Since the inexact solution $\mathbf{V}^{k+1}$ at $k^{th}$ iteration of s-iPDCA  satisfies  the inexact optimality condition \eqref{eq_13}, expressed as
\begin{equation} \label{eq_26_1}
0 \in \frac{2}{n} \mathcal{A}^*\left(\mathcal{A}\left(\mathbf{V}^{k+1}\right) - \boldsymbol{b}\right) -\mathbf{\Phi}^{k}+\partial\delta_{\mathcal{S}^d_+}\left(\mathbf{V}^{k+1}\right)+\alpha\mathbf{V}^{k+1}-\mathbf{\Delta}^{k+1},
\end{equation}
 in which, $\|\mathbf{\Delta}^{k+1}\|_F\leq\epsilon_{k+1}$. As we known,  \eqref{eq_26_1} is equivalent to
\begin{equation} \label{eq_27_1}
\mathbf{V}^{k+1} =  \operatorname{prox}_{\delta_{\mathcal{S}^d_+}}\left(\left(1-\alpha\right)\mathbf{V}^{k+1}-\frac{2}{n} \mathcal{A}^*\left(\mathcal{A}\left(\mathbf{V}^{k+1}\right) - \boldsymbol{b}\right) +\mathbf{\Phi}^{k}+\mathbf{\Delta}^{k+1}\right).
\end{equation}
In here, $\operatorname{prox}_{\delta_{\mathcal{S}^d_+}}$ denotes the proximal projection of $\delta_{\mathcal{S}^d_+}$, which is just the  positive semidefinite matrix cone projection, given as
\begin{equation}\label{eq_28_1}
\mathbf{V}^{k+1} =  \mathbf{\Pi}_{\mathcal{S}^d_+}\left(\left(1-\alpha\right)\mathbf{V}^{k+1}-\frac{2}{n} \mathcal{A}^*\left(\mathcal{A}\left(\mathbf{V}^{k+1}\right) - \boldsymbol{b}\right) +\mathbf{\Phi}^{k}+\mathbf{\Delta}^{k+1}\right)
\end{equation}
To obtain the above inexact solution $\mathbf{V}^{k+1}$ and  inexact term $\mathbf{\Delta}^{k+1}$,  we set
\begin{equation}\label{eq_29_1}
\tilde{\mathbf{U}}^{(j)} =  \mathbf{\Pi}_{\mathcal{S}^d_+}\left(\left(1-\alpha\right)\mathbf{U}^{(j)}-\frac{2}{n} \mathcal{A}^*\left(\mathcal{A}\left(\mathbf{U}^{(j)}\right) - \boldsymbol{b}\right) +\mathbf{\Phi}^{k}\right),
\end{equation}
in here, $\mathbf{U}^{(j)}=\frac{\mathbf{Y}^j+\mathbf{\Phi}^{k}-\mathcal{A}^*\left(\boldsymbol{z}^j\right)}{\alpha}$ is a feasible solution of \eqref{eq_16}. Then we know that $\mathbf{V}^{k+1} = \tilde{\mathbf{U}}^{(j)}$ is just the inexact solution  when \[\left\|(\frac{2}{n} \mathcal{A}^*\mathcal{A}+(\alpha-1)\mathcal{I})(\tilde{\mathbf{U}}^{(j)}-\mathbf{U}^{(j)})\right\|\leq \epsilon_{k+1}\] holds, and \[\mathbf{\Delta}^{k+1} = \left(\frac{2}{n} \mathcal{A}^*\mathcal{A}+(\alpha-1)\mathcal{I}\right)\left(\tilde{\mathbf{U}}^{(j)}-\mathbf{U}^{(j)}\right)\] is just the inexact term corresponding to $\mathbf{V}^{k+1}$.\par
Secondly, we will show the relation between $\mathbf{\Delta}^{k+1}$ and $\gamma^j$.
As defined in the first subsection, the solutions $\left(\boldsymbol{z}^j,\mathbf{Y}^j\right)$ of ABCD \ref{alg_abcd} at $j^{th}$ iteration satisfy the following equation:
\begin{equation} \label{eq_30_1}
\gamma^j=\frac{n}{2}\boldsymbol{z^j} + \boldsymbol{b}+\frac{1}{\alpha}\mathcal{A}\left(\mathcal{A}^*\left(\boldsymbol{z}^j\right)-\mathbf{Y}^j-\mathbf{\Phi}^{k}\right).
\end{equation}
 From $\mathbf{U}^{(j)}=\frac{\mathbf{Y}^j+\mathbf{\Phi}^{k}-\mathcal{A}^*\left(\boldsymbol{z}^j\right)}{\alpha}$, we have
\begin{equation}\label{eq_31_1}
-\frac{2}{n} \mathcal{A}^*\left(\mathcal{A}\left(\mathbf{U}^{(j)}\right) -\boldsymbol{b}\right)= \mathcal{A}^*\left(\frac{2}{n}\gamma^j-\boldsymbol{z}^j\right).
\end{equation}
Furthermore, the projection  $\tilde{\mathbf{U}}^{(j)}$ can be reformulated as
\begin{equation}\nonumber
\begin{array}{ll}
\tilde{\mathbf{U}}^{(j)} 
&= \mathbf{\Pi}_{\mathcal{S}^d_+}\left(\mathbf{U}^{(j)}+\mathcal{A}^*\left(\frac{2}{n}\gamma^j-\boldsymbol{z}^j\right) +\mathbf{\Phi}^{k}-\left(\mathbf{Y}^j+\mathbf{\Phi}^{k}-\mathcal{A}^*\left(\boldsymbol{z}^j\right)\right)\right)\\
&= \mathbf{\Pi}_{\mathcal{S}^d_+}\left(\mathbf{U}^{(j)}+\frac{2}{n}\mathcal{A}^*\left(\gamma^j\right) -\mathbf{Y}^j\right).
\end{array}
\end{equation}
In here, the second euqality follows from \eqref{eq_31_1}.
Clearly, \[\tilde{\mathbf{U}}^{(j)}= \mathbf{\Pi}_{\mathcal{S}^d_+}\left(\mathbf{U}^{(j)} -\mathbf{Y}^j\right) = \mathbf{U}^{(j)}\] holds when $\gamma^j = \boldsymbol{0}$, which is consistent with intuition that solving the dual problem exactly leads to exact solution of primal problem $(\mathbf{\Delta}^{(j)} = \mathbf{0})$. When  $\gamma^j \neq \boldsymbol{0}$, the inexact term $\mathbf{\Delta}^{(j)}$ can be denoted as
\begin{equation}\label{eq_32_1}
\begin{array}{lll}
\mathbf{\Delta}^{(j)} 
&= \left(\frac{2}{n} \mathcal{A}^*\mathcal{A}+(\alpha-1)\mathcal{I}\right)\left(\tilde{\mathbf{U}}^{(j)}-\mathbf{U}^{(j)}\right)\\
&= \left(\frac{2}{n} \mathcal{A}^*\mathcal{A}+(\alpha-1)\mathcal{I}\right)\left(\mathbf{\Pi}_{\mathcal{S}^d_+}\left(\mathbf{U}^{(j)}+\frac{2}{n}\mathcal{A}^*\left(\gamma^j\right) -\mathbf{Y}^j\right)-\mathbf{U}^{(j)}\right)
\end{array}
\end{equation}
From the non-expansibility of the prox projection operator, we have
\begin{equation}\label{eq_33_1}
\begin{array}{lll}
&\|\operatorname{prox}_{\delta_{\mathcal{S}^d_+}}(\mathbf{U}^{(j)}+\frac{2}{n}\mathcal{A}^*(\gamma^j) -\mathbf{Y}^j)-\mathbf{U}^{(j)}\|_F\\
&=\|\operatorname{prox}_{\delta_{\mathcal{S}^d_+}}(\mathbf{U}^{(j)}+\frac{2}{n}\mathcal{A}^*(\gamma^j) -\mathbf{Y}^j)-\operatorname{prox}_{\delta_{\mathcal{S}^d_+}}(\mathbf{U}^{(j)})\|_F\\
&=\|\operatorname{prox}_{\delta_{\mathcal{S}^d_+}}(\mathbf{U}^{(j)}+\frac{2}{n}\mathcal{A}^*(\gamma^j) -\mathbf{Y}^j)-\operatorname{prox}_{\delta_{\mathcal{S}^d_+}}(\mathbf{U}^{(j)}-\mathbf{Y}^j)\|_F\leq \frac{2}{n}\|\mathcal{A}^*(\gamma^j)\|_F.
\end{array}
\end{equation}
In here, the second equality from the \eqref{eq_23_1} and the relation $\mathbf{U}^{(j)}=\frac{\mathbf{Y}^j+\mathbf{\Phi}^{k}-\mathcal{A}^*\left(\boldsymbol{z}^j\right)}{\alpha}$. Then we have 
\begin{equation}\nonumber
\begin{array}{lll}
\|\mathbf{\Delta}^{(j)}\|_F
&\leq\frac{2}{n}\|\frac{2}{n} \mathbf{A}^{\top}\mathbf{A}+(\alpha-1)\mathbf{I}\|_F\|\mathcal{A}^*(\gamma^j)\|_F\\
&\leq \frac{2}{n}\|\mathbf{A}\|_F\|\frac{2}{n} \mathbf{A}^{\top}\mathbf{A}+(\alpha-1)\mathbf{I}\|_F\|\gamma^j\|.
\end{array}
\end{equation}
Then we know that $\mathbf{V}^{k+1} = \tilde{\mathbf{U}}^{(j)}$ is just the solution of \eqref{eq_16} satisfied inexact condition $\|\mathbf{\Delta}^{(j)}\|_F< \epsilon_{k+1}$ when $\|\gamma^j\|< \zeta_k$ holds true, where $\zeta_k$ is defined as 
\begin{equation}\label{eq_34_1}
\zeta_k := \frac{n}{2}\|\mathbf{A}\|_F^{-1}\|\frac{2}{n} \mathbf{A}^{\top}\mathbf{A}+(\alpha-1)\mathbf{I}\|_F^{-1}\epsilon_{k+1} .
\end{equation} It should be noted that the cost of computing the norm $\|\frac{2}{n} \mathbf{A}^{\top}\mathbf{A}+(\alpha-1)\mathbf{I}\|_F$ directly is  expensive when $d^2>>p$. Thus, we compute this norm as following:
\[\|\frac{2}{n} \mathbf{A}^{\top}\mathbf{A}+(\alpha-1)\mathbf{I}\|_F = \sqrt{\|\frac{2}{n} \mathbf{A}\mathbf{A}^{\top}\|_F^2+\|(1-\alpha)\mathbf{I}\|_F^2 -\langle (1-\alpha)\mathbf{A}\mathbf{A}^{\top}, \frac{2}{n}\mathbf{I}\rangle}\]
which reduces the computing cost from $O(d^4)$ to $O(p^2)$. 
\par
Next, we shall study the convergence of the proposed s-iPDCA for solving the DCLSSDP \eqref{eq_11}.

\section{Global convergence analysis for s-iPDCA}\label{sec:4}
 In this section, we present the convergence for the proposed s-iPDCA for solving \eqref{eq_11}. A feasible point $\mathbf{U}\in \mathcal{S}_+^d$ is said to be a stationary point of the DC problem \eqref{eq_11} if
\begin{equation}\label{eq_35_1}
\left(\nabla f_1\left(\mathbf{U}\right)+\mathcal{N}_{\mathcal{S}_+^d}\left(\mathbf{U}\right)\right)\bigcap \partial f_2\left(\mathbf{U}\right)\neq\phi
\end{equation} 
where $\mathcal{N}_{\mathcal{S}_+^d}\left(U\right)$ is the normal cone of the convex set $\mathcal{S}_+^d$ at $\mathbf{U}$, that is equal to $\partial\delta_{\mathcal{S}_+^d}\left(U\right)$ because of the convexity of $\mathcal{S}_+^d$ \cite{Ref_rockafellar1970convex}. The following results on the convergence of the proposed s-iPDCA (Algorithm \ref{alg_s_iPDCA}) for the  DCLSSDP \eqref{eq_11} follows from the basic convergence theorem of DCA\cite{Ref_tao1997convex}.
\par
To ensure the objective function in \eqref{eq_11} is coercive,  we assume $\mathcal{A}$ satisfies the Restricted Isometry Property (RIP) condition\cite{Ref_candes2008restricted}, which is also used as one of the most standard assumptions in the low-rank matrix recovery literatures\cite{Ref_bhojanapalli2016global,Ref_cai2014sparse,Ref_luo2020recursive}.
\begin{definition}[Restricted Isometry Property (RIP)]\label{def_1}
 Let $\mathcal{A}:\mathcal{S}_+^d\rightarrow \mathcal{R}^p$ be a linear map.
For every integer $r$ with $1\leq r\leq d$, define the r-restricted isometry constant to be the
smallest number $R_r$ such that
\[\left(1-R_{r}\right)\|\mathbf{U}\|_F^{2} \leq \|\mathcal{A}(\mathbf{U})\|^{2} \leq\left(1+R_{r}\right)\|\mathbf{U}\|_F^{2}\]
holds for all $\mathbf{U}$ of rank at most r. And $\mathcal{A}$ is said to satisfy the r-restricted isometry property (r-RIP) if $0 \leq R_{r} <  1$.
\end{definition}

\begin{proposition}\label{prop_4} Assume $\mathcal{A}$ satisfies r-restricted isometry property (r-RIP). The sequence of stability center $\{\mathbf{U}^{k}\}$ is generated by s-iPDCA for solving \eqref{eq_11}, the following statements hold.\\
$(1)$. The $J_c$ in \eqref{eq_11} is lower bounded and coercive.\\
$(2)$. The sequence $\left\{J_c(\mathbf{U}^{k})\right\}$ is nonincerasing.\\
$(3)$. The sequence $\left\{\mathbf{U}^{k}\right\}$ is bounded.
\end{proposition}
\begin{proof}
First we prove (1). As is shown in \eqref{eq_11}, the $J_c$ is the sum of two non-negative functions: $\frac{1}{n}\|\mathcal{A}(\mathbf{U})- \boldsymbol{b}\|^2$ and $c\left(\langle \mathbf{U},\mathbf{I}\rangle-\|\mathbf{U}\|_{(r)}\right)$, so $J_c$ is lower bounded. Combining this with the r-RIP of $\mathcal{A}$, we have
\begin{equation}\nonumber
\begin{array}{lll}
 J_c& = \frac{1}{n}\|\mathcal{A}(\mathbf{U})- \boldsymbol{b}\|^2 + c\left(\langle \mathbf{U},\mathbf{I}\rangle-\|\mathbf{U}\|_{(r)}\right)\geq \frac{1}{n}\|\mathcal{A}(\mathbf{U})- \boldsymbol{b}\|^2\\
 &\geq \frac{1}{n}\|\mathcal{A}(\mathbf{U})\|^2-\frac{2}{n}\| \mathcal{A}(\mathbf{U})\|\|\boldsymbol{b}\|+\frac{1}{n}\|\boldsymbol{b}\|^2\\
  &\geq \frac{1}{n}(1-R_r)^2\|\mathbf{U}\|_F^2-\frac{2}{n}(1+R_r)\|\mathbf{U}\|_F\|\boldsymbol{b}\|+\frac{1}{n}\|\boldsymbol{b}\|^2.
\end{array}
\end{equation}
So the $J_c$ in \eqref{eq_11} trends to infinity only if $\|\mathbf{U}\|_F$ trends to infinity, which implies that $J_c$  is coercive. \par
For state (2). If $\mathbf{U}^{k+1}$ is generated in null step, namely, $\mathbf{U}^{k+1} = \mathbf{U}^{k}$, the $J_c\left(\mathbf{U}^{k+1}\right)\leq J_c\left(\mathbf{U}^{k}\right)$ holds immediately.
In the another situation,  if  $\mathbf{U}^{k+1} = \mathbf{V}^{k+1}$, then from the optimality of $\mathbf{U}^{k+1}$ for solving \eqref{eq_16} and the feasibility of $\mathbf{U}^{k}$, we have
\begin{equation}\label{eq_36}
\begin{array}{ll}
&f_1(\mathbf{U}^{k+1})-\langle \mathbf{U}^{k+1},\mathbf{W}^{k}\rangle+\frac{\alpha}{2}\|\mathbf{U}^{k+1}-\mathbf{U}^{k}\|_F^2-\langle\mathbf{U}^{k+1},\mathbf{\Delta}^{k+1}\rangle\\
&\leq f_1(\mathbf{U}^{k})-\langle\mathbf{U}^{k},\mathbf{W}^{k}\rangle -\langle\mathbf{U}^{k},\mathbf{\Delta}^{k+1}\rangle.
\end{array}
\end{equation}
Thanks to the convexity of $f_2\left(\mathbf{U}\right)$, we have 
\[f_2\left(\mathbf{U}^{k+1}\right)\geq f_2\left(\mathbf{U}^{k}\right) + \langle\mathbf{U}^{k+1}-\mathbf{U}^{k},\mathbf{W}^{k}\rangle.\]
Combining this with \eqref{eq_36}, we get
\begin{equation}\label{eq_37}
\begin{array}{ll}
&\frac{\alpha}{2}\|\mathbf{U}^{k+1}-\mathbf{U}^{k}\|_F^2-\langle\mathbf{U}^{k+1}-\mathbf{U}^{k},\mathbf{\Delta}^{k+1}\rangle\\
&\leq \left[f_1(\mathbf{U}^{k})-f_2(\mathbf{U}^{k})\right]-\left[f_1(\mathbf{U}^{k+1})-f_2(\mathbf{U}^{k+1})\right].
\end{array}
\end{equation}
Since $\mathbf{U}^{k+1}= \mathbf{V}^{k+1}$ is generated in serious step of s-iPDCA, so the  test \eqref{eq_15} holds true:
\begin{equation}\label{eq_38} 
\|\mathbf{\Delta}^{k+1}\|_F\leq\frac{(1-\kappa)\alpha}{2}\|\mathbf{U}^{k+1}-\mathbf{U}^{k}\|_F,
\end{equation}
Then we have
\begin{equation}\label{eq_39}
\begin{array}{ll}
&\frac{\alpha}{2}\|\mathbf{U}^{k+1}-\mathbf{U}^{k}\|_F^2-\langle\mathbf{U}^{k+1}-\mathbf{U}^{k},\mathbf{\Delta}^{k+1}\rangle\\
 &\geq \frac{\alpha}{2}\|\mathbf{U}^{k+1}-\mathbf{U}^{k}\|_F^2-\|\mathbf{U}^{k+1}-\mathbf{U}^{k}\|_F\|\mathbf{\Delta}^{k+1}\|_F\geq \frac{\kappa\alpha}{2}\|\mathbf{U}^{k+1}-\mathbf{U}^{k}\|_F^2.
 \end{array}
\end{equation}
Applying this to \eqref{eq_37}, we get
\begin{equation}\label{eq_40}
\begin{array}{lll}
\frac{\kappa\alpha}{2}\|\mathbf{U}^{k+1}-\mathbf{U}^{k}\|_F^2\leq\left[f_1\left(\mathbf{U}^{k}\right) - f_2\left(\mathbf{U}^{k}\right)\right] -\left[f_1\left(\mathbf{U}^{k+1}\right) - f_2\left(\mathbf{U}^{k+1}\right)\right].
\end{array}
\end{equation}
So the sequence $\left\{f_1(\mathbf{U}^{k}) - f_2(\mathbf{U}^{k})\right\}$ is nonincerasing. \par
 For state (3). Since $J_c(\mathbf{U}^0)<\infty$, then according to the results in  state (1)  and state (2), we know $\left\{J_c(\mathbf{U}^k)\right\}$ is bounded. If $\left\{\mathbf{U}^{k}\right\}$ is unbounded, namely, there exists subset $\mathcal{K}^{\prime}\subset\mathcal{K}= \left\{0,1,2,\cdots\right\}$ such that $\|\mathbf{U}^k\|_F= \infty$ for $k\in \mathcal{K}^{\prime}$. According to coercive property of $J_c$, then $J_c(\mathbf{U}^k)= \infty$ for $k\in \mathcal{K}^{\prime}$ holds, which is contrary to the boundedness of $\left\{J_c(\mathbf{U}^k)\right\}$. Thus, the boundness of the sequence $\left\{\mathbf{U}^{k}\right\}$ holds true. This completes the proof.
\end{proof}
Next, set the tolerance error $\varepsilon = 0$, we divide our convergence analysis  into two parts: firstly, we consider the case that  only finite serious steps are performed in s-iPDCA for solving \eqref{eq_11}; secondly, we suppose that infinite serious steps are performed in s-iPDCA for solving \eqref{eq_11}.
\subsection{Finite serious steps in s-iPDCA}\label{sec:4_1}
In this subsection, we provide the convergence analysis for the situation that  only finite serious steps are performed in s-iPDCA for solving \eqref{eq_11}.
\begin{proposition}\label{prop_5}
Set the tolerance error $\varepsilon = 0$. Suppose that only finite serious steps are performed in s-iPDCA for solving \eqref{eq_11}. The following statements hold:\\
$(1)$. If the algorithm s-iPDCA  is terminated in finite steps, namely, there exists  $\bar{k}>0$ such that $\mathbf{V}^{\bar{k}+1}= \mathbf{U}^{\bar{k}}$, then the stability center $\mathbf{U}^{\bar{k}}$ is an $\epsilon_1$-stationary point of \eqref{eq_11}.\\
$(2)$. If after $\hat{k}^{th}$ iteration of s-iPDCA \ref{alg_s_iPDCA}, only null step is performed, namely, $\mathbf{W}^{k+1} = \mathbf{W}^{\hat{k}+1}$ and $ \mathbf{U}^{k+1} = \mathbf{U}^{\hat{k}+1}$,  $\forall k>\hat{k}$, then 
\begin{equation} \label{eq_41}
\lim_{k\rightarrow \infty}\mathbf{V}^{k+1} = \mathbf{U}^{\hat{k}+1},
\end{equation}
and the  stability center $\mathbf{U}^{\hat{k}+1}$ generated in the last serious step is a stationary point of \eqref{eq_11}.
\end{proposition}
\begin{proof}
For state (1).  From the optimality of $\mathbf{V}^{\bar{k}+1}$ for solving the strongly convex subproblem \eqref{eq_16}, we have that $\mathbf{U}^{\bar{k}} =\mathbf{V}^{\bar{k}+1}$ solves
\[\min  f_1\left(\mathbf{U}\right)- \langle \mathbf{U}, \mathbf{W}^{\bar{k}}\rangle +\delta_{\mathcal{S}^m_+}\left(\mathbf{U}\right) + \frac{\alpha}{2}\|\mathbf{U}-\mathbf{U}^{\bar{k}}\|_F^2-\langle \mathbf{\Delta}^{\bar{k}+1}, \mathbf{U}\rangle.\]
Thus, the optimality condition
\[0 \in \nabla f_1\left(\mathbf{U}^{\bar{k}}\right) -\mathbf{W}^{\bar{k}}+\partial\delta_{\mathcal{S}^m_+}\left(\mathbf{U}^{\bar{k}}\right)-\mathbf{\Delta}^{\bar{k}+1}\]
holds true. Then, there exists $\xi\in \partial\delta_{\mathcal{S}^m_+}\left(\mathbf{U}^{\bar{k}}\right)$ such that
\[ \nabla f_1\left(\mathbf{U}^{\bar{k}}\right) - \mathbf{W}^{\bar{k}}+\xi -\mathbf{\Delta}^{\bar{k}+1} = \mathbf{0}.\] 
Thus, $\|\nabla f_1\left(\mathbf{U}^{\bar{k}}\right) - \mathbf{W}^{\bar{k}}+\xi \|_F= \|\mathbf{\Delta}^{\bar{k}+1}\|_F\leq \epsilon_{\bar{k}+1}\leq\epsilon_{1}$. So we have that 
\[\mathbf{0} \in \partial_{\epsilon_1}\left[f_1\left(\mathbf{U}^{\bar{k}}\right) - f_2\left(\mathbf{U}^{\bar{k}}\right)+\delta_{\mathcal{S}^m_+}\left(\mathbf{U}^{\bar{k}}\right)\right],\]
which means that $\mathbf{0}$ belongs to the $\epsilon_1$-inexact subdifferential of $f_1\left(\mathbf{U}\right)- f_2\left(\mathbf{U}\right)+\delta_{\mathcal{S}^m_+}\left(\mathbf{U}\right)$ at point $\mathbf{U}^{\bar{k}}$. Then $\mathbf{U}^{\bar{k}}$ is an $\epsilon_1$-stationary point of DC problem \eqref{eq_11}.

For state (2). We first show $\lim_{k\rightarrow \infty}\mathbf{V}^{k+1} = \mathbf{U}^{\hat{k}}$. Since the test \eqref{eq_15} doesn't hold for all $k>\hat{k}$, that's to say
\begin{equation}\label{eq_42}
\frac{\left(1-\kappa\right)\alpha}{2}\|\mathbf{V}^{k+1}-\mathbf{U}^{\hat{k}}\|_F^2\leq \|\mathbf{\Delta}^{k+1}\|_F\leq \epsilon_{k+1}
\end{equation}
holds true,  $\forall k>\hat{k}$.
 Thanks to the monotonic descent property of the sequence $\left\{\epsilon_{k+1}\right\}$ and  $\lim_{k\rightarrow\infty}\epsilon_{k+1} = 0$, take limit on both sides of inequality \eqref{eq_42}, we obtain
 \begin{equation}\label{eq_43}
\lim_{k\rightarrow \infty}\frac{\left(1-\kappa\right)\alpha}{2}\|\mathbf{V}^{k+1}-\mathbf{U}^{\hat{k}}\|_F^2 = 0.
\end{equation}
So  $\lim_{k\rightarrow \infty}\mathbf{V}^{k+1} = \mathbf{U}^{\hat{k}}$. Next, we show that $\mathbf{U}^{\hat{k}}$ is a stationary point of \eqref{eq_11}. Due to the optimality of $\mathbf{V}^{k+1}$ for solving \eqref{eq_14}, we know that $\forall k > \hat{k}$, 
\[0 \in \nabla f_1\left(\mathbf{V}^{k+1}\right) - \mathbf{W}^{\hat{k}}+\partial\delta_{\mathcal{S}^m_+}\left(\mathbf{V}^{k+1}\right)+\alpha\left(\mathbf{V}^{k+1} -\mathbf{U}^{\hat{k}}\right)-\mathbf{\Delta}^{k+1}.\]
Then, there  exists $\zeta^{k+1}\in\partial\delta_{\mathcal{S}^m_+}\left(\mathbf{V}^{k+1}\right)$ such that
\begin{equation}\label{eq_44}
 \nabla f_1\left(\mathbf{V}^{k+1}\right) - \mathbf{W}^{\hat{k}}+\zeta^{k+1}+\alpha\left(\mathbf{V}^{k+1} -\mathbf{U}^{\hat{k}}\right)-\mathbf{\Delta}^{k+1} = 0.
\end{equation}
So we have
 \begin{equation}\label{eq_45} 
\begin{array}{ll}
\|\nabla f_1(\mathbf{V}^{k+1}) - \mathbf{W}^{\hat{k}}+\zeta^{k+1}\|_F&\leq \alpha\|\mathbf{V}^{k+1} -\mathbf{U}^{\hat{k}}\|_F+\|\mathbf{\Delta}^{k+1}\|_F\\
&<\alpha\|\mathbf{V}^{k+1} -\mathbf{U}^{\hat{k}}\|_F+\epsilon_{k+1}.
 \end{array}
\end{equation}
According to proposition \ref{prop_4},  $\left\{\mathbf{U}^{k}\right\}$ is bounded,  we know that $\left\{\mathbf{U}^{k}\right\}$ is bounded from \[\lim_{k\rightarrow\infty}\mathbf{V}^{k+1} = \mathbf{U}^{\hat{k}}.\] Then the boundness of $\left\{\zeta^{k}\right\}$  can be deduced from the convexity and continuity of $\partial\delta_{\mathcal{S}^m_+}$. Thus, there exists subset $\mathcal{K}^{\prime}\subset\mathcal{K} = \left\{0,1,2,\cdot\cdot\cdot\right\}$ such that $\lim_{k\in\mathcal{K}^{\prime}}\zeta^{k+1} = \hat{\zeta}\in \partial\delta_{\mathcal{S}^m_+}\left(\mathbf{U}^{\hat{k}}\right) $ and $\lim_{k\in\mathcal{K}^{\prime}}\epsilon_k = 0$. Then
 \begin{equation}\label{eq_46} 
\lim_{k\in\mathcal{K}^{\prime}}\|\nabla f_1\left(\mathbf{V}^{k+1}\right) - \mathbf{W}^{\hat{k}}+\zeta^{k+1}\|_F=\|\nabla f_1\left(\mathbf{U}^{\hat{k}}\right) - \mathbf{W}^{\hat{k}}+\hat{\zeta}\|_F=0.
\end{equation}
So we have
\begin{equation} \label{eq_47}
0 \in \nabla f_1\left(\mathbf{U}^{\hat{k}}\right) -\partial f_2\left(\mathbf{U}^{\hat{k}}\right) +\partial\delta_{\mathcal{S}^m_+}\left(\mathbf{U}^{\hat{k}}\right),
\end{equation}
which implies that $\mathbf{U}^{\hat{k}}$ is a stationary point of problem \eqref{eq_11}. This completes the proof.
\end{proof}
\subsection{Infinite serious steps in s-iPDCA}\label{sec:4_2}
In this subsection, we consider the case that  infinite serious steps are performed in s-iPDCA for solving \eqref{eq_11} when tolerance error $\varepsilon $ is set to $0$.

\begin{theorem}[Global subsequential convergence of s-iPDCA]\label{thm:1}
Set the tolerance error $\varepsilon = 0 $. Let $\{\mathbf{U}^{k}\}$ be the  stablity center sequence generated by s-iPDCA for solving \eqref{eq_11}. Then the following statements hold:\\
$(1)$. $\lim_{k\rightarrow \infty}\|\mathbf{U}^{k+N}-\mathbf{U}^{k}\|_F =0$,  $\forall N\in \mathcal{N}^+, N<\infty$.\\
$(2)$. Any accumulation point  $\bar{\mathbf{U}}\in\left\{\mathbf{U}^{k}\right\}$ is stationary point of \eqref{eq_11}.
\end{theorem}
\begin{proof}
For state (1). From Proposition \ref{prop_4}, we know that the sequence $\left\{f_1(\mathbf{U}^{k}) - f_2(\mathbf{U}^{k})\right\}$ is nonincerasing and bounded below, so $\liminf_{k\rightarrow \infty}\left[f_1\left(\mathbf{U}^{k+1}\right) -f_2\left(\mathbf{U}^{k+1}\right)\right]<\infty$. Thus, by summing both sides of \eqref{eq_40} from $k = 0$ to $\infty$, we obtain that
\begin{equation}\label{eq_48}
\begin{array}{lll}
&\mathbf{\Sigma}_{k= 0}^{\infty}\frac{\kappa\alpha}{2}\|\mathbf{U}^{k+1}-\mathbf{U}^{k}\|_F^2\\
&\leq\left[f_1(\mathbf{U}^{0}) - f_2(\mathbf{U}^{0})\right] -\liminf_{k\rightarrow \infty}\left[f_1(\mathbf{U}^{k+1}) - f_2(\mathbf{U}^{k+1})\right]<\infty.
\end{array}
\end{equation}
Then we deduce that $\left\{\mathbf{U}^{k}\right\}$ is Cauchy sequence, so $\lim_{k\rightarrow \infty}\|\mathbf{U}^{k+N}-\mathbf{U}^{k}\|_F = 0$, $\forall N\in \mathcal{N}^+, N<\infty$.\par
For state (2). 
Since the number of serious steps is infinite, then $\forall k>0$, there exists  $N\in \mathcal{N}^+, N<\infty$, such that $\mathbf{U}^{k+N}$ is the  stability center generated at the new serious step of s-iPDCA \ref{alg_s_iPDCA}, so the following optimality condition for solving \eqref{eq_16} holds true:
\begin{equation} \label{eq_49}
0 \in \nabla f_1\left(\mathbf{U}^{k+N}\right) - \mathbf{W}^{k}+\partial\delta_{\mathcal{S}^m_+}\left(\mathbf{U}^{k+N}\right)+\alpha\left(\mathbf{U}^{k+N} -\mathbf{U}^{k}\right)-\mathbf{\Delta}^{k+N}.
\end{equation}
Then, there  exists $\zeta^{k+N}\in\partial\delta_{\mathcal{S}^m_+}\left(\mathbf{U}^{k+N}\right)$ such that 
\[\nabla f_1\left(\mathbf{U}^{k+N}\right) - \mathbf{W}^{k}+\zeta^{k+N}+\alpha\left(\mathbf{U}^{k+N} -\mathbf{U}^{k}\right)-\mathbf{\Delta}^{k+N}=0.\]
So we have
\begin{equation}\label{eq_50}
\left\|\nabla f_1(\mathbf{U}^{k+N}) - \mathbf{W}^{k}+\zeta^{k+N}\right\|_F\leq \alpha\left\|\mathbf{U}^{k+N} -\mathbf{U}^{k}\right\|_F+\epsilon_{k+N}.
\end{equation}
Due to the boundness of $\left\{\mathbf{U}^{k}\right\}$, there exists subsets $\mathcal{K}^{\prime} \subset \mathcal{K}=\{0,1,2 \ldots\}$ such that $\left\{\mathbf{U}^{k}\right\}_{\mathcal{K}^{\prime}}$ converges to an accumulation point $\bar{\mathbf{U}}\in\left\{\mathbf{U}^{k}\right\}_{\mathcal{K}}$. 
Combining  the boundedness of $\left\{\mathbf{U}^{k}\right\}_{\mathcal{K}^{\prime}}$ with the continuity and convexity of $f_2$ and $\delta_{\mathcal{S}^m_+}$, we deduce that the subsequence $\left\{\mathbf{W}^{k}\right\}_{\mathcal{K}^{\prime}}$  and $\left\{\zeta^{k+N}\right\}_{\mathcal{K}^{\prime}}$ are bounded. By the fact that the nonnegative sequence $\left\{\epsilon_{k+N}\right\}$   monotonically decreases to zero, we may assume without loss of generality that there exist  a subset $\mathcal{K}^{\prime \prime} \subset \mathcal{K}^{\prime}$ such that $\lim_{k\in \mathcal{K}^{\prime \prime}}\mathbf{W}^{k} = \bar{\mathbf{W}}\in \partial f_2(\bar{\mathbf{U}})$, $\lim_{k\in \mathcal{K}^{\prime \prime}}\zeta^{k+N} = \bar{\zeta}\in \partial\delta_{\mathcal{S}^m_+}\left(\bar{\mathbf{U}}\right)$ and $\lim_{k\in \mathcal{K}^{\prime \prime}}\epsilon_{k+N} = 0$.
Taking the limit on the two sides of inequality in \eqref{eq_50} with $k\in \mathcal{K}^{\prime \prime}$, we have
\begin{equation}\label{eq_51}
\begin{array}{ll}
\left\|\nabla f_1\left(\bar{\mathbf{U}}\right) - c\bar{\mathbf{W}}+\bar{\zeta}\right\|_F&=\lim_{k\in\mathcal{K}^{\prime\prime}}\left\|\nabla f_1(\mathbf{U}^{k+N}) - \mathbf{W}^{k}+\zeta^{k+N}\right\|_F\\
& \leq \lim_{k\in\mathcal{K}^{\prime\prime}} \alpha\left\|\mathbf{U}^{k+N} -\mathbf{U}^{k}\right\|_F+\lim_{k\in \mathcal{K}^{\prime \prime}}\epsilon_{k+N}= 0,
\end{array}
\end{equation}
which implies that $\left\|\nabla f_1\left(\bar{\mathbf{U}}\right) - c\bar{\mathbf{W}}+\bar{\zeta}\right\|_F = 0$. 
Therefore, any accumulation point $\bar{\mathbf{U}}\in\left\{\mathbf{U}^{k}\right\}$ satisfies the following optimality condition:
\begin{equation} \label{eq_52}
0 \in \nabla f_1\left(\bar{\mathbf{U}}\right) -\partial f_2\left(\bar{\mathbf{U}}\right) +\partial\delta_{\mathcal{S}^m_+}\left(\bar{\mathbf{U}}\right).
\end{equation}
This implies that any accumulation point of $\left\{\mathbf{U}^{k}\right\}$ is a stationary point of \eqref{eq_11}.
This completes the proof.
\end{proof}

In order to show that the sequence $\{\mathbf{U}^k\}$  actually converges to the stationary points of \eqref{eq_11} when infinite serious steps are performed in our Algorithm \ref{alg_s_iPDCA}, we recall the following definition of the Kurdyka-Łojaziewicz (KL) property of the lower semi-continuous function\cite{Ref_attouch2009convergence,Ref_bolte2016majorization,Ref_bolte2014proximal}. Let $a \geq 0$ and $\mathbf{\Xi}_a$ be the class of functions $\varphi:[0,a)\rightarrow \mathcal{R}^+$ that satisfy
the following conditions:\\
$(1)$.$\varphi(0) = 0$\\
$(2)$.$\varphi$ is positive, concave and continuous\\
$(3)$.$\varphi$ is continuously differentiable on $\left[0,a\right)$ with  $\varphi^{\prime}(x)>0$,  $\forall x\in \left[0,a\right)$.

\begin{definition}[KL property]\label{def_3}
The given proper lower semicontinuous function
$h:\mathcal{R}^q \rightarrow (-\infty, \infty]$ is said to have the KL property at $\bar{\boldsymbol{u}} \in \operatorname{dom} h$ if there exist $a > 0$ ,
a neighborhood $\mathcal{U}$ of $\bar{\boldsymbol{u}}$ and a concave function $\varphi\in\mathbf{\Xi}_a$ such that
\[\varphi^{\prime}\left(h\left(\boldsymbol{u}\right)-h\left(\bar{\boldsymbol{u}}\right)\right)\operatorname{dist}\left(0,\partial h\left(\boldsymbol{u}\right)\right)\geq 1,\forall \boldsymbol{u}\in\mathcal{U} \quad\text{and} \quad h\left(\bar{\boldsymbol{u}}\right)\leq h\left(\boldsymbol{u}\right) \leq h\left(\bar{\boldsymbol{u}}\right)+a,\]
﻿in here, $\operatorname{dist}\left(\boldsymbol{0},\partial h\left(\boldsymbol{u}\right)\right)$ is the distance from the point $\boldsymbol{0}$ to a nonempty closed set $\partial h\left(\boldsymbol{u}\right)$. The function h is said to be a KL function if it has the KL property at each point of $\operatorname{dom}h$.
\end{definition}

\begin{lemma}[Uniformized KL property]\label{lemma_1}
Suppose that h is a proper closed function and let $\mathbf{\Gamma}$ be a compact set. If h is a constant on $\mathbf{\Gamma}$ and satisfies the KL property at each point of  $\mathbf{\Gamma}$, then there exists $\epsilon>0$, $a>0$ and $\varphi\in \mathbf{\Xi}_a$ such that 
\begin{equation}\label{eq_53}
\varphi^{\prime}\left(h\left(\boldsymbol{u}\right)-h\left(\bar{\boldsymbol{u}}\right)\right)\operatorname{dist}\left(0,\partial h\left(\boldsymbol{u}\right)\right)\geq 1
\end{equation}
for any $\bar{\boldsymbol{u}}\in\mathbf{\Gamma}$ and  $\boldsymbol{u}\in\mathcal{U}$ with 
\[\mathcal{U} = \left\{\boldsymbol{u}|\operatorname{dist}\left(0,\partial h\left(\boldsymbol{u}\right)\right)<\epsilon \quad \text{and} \quad h\left(\bar{\boldsymbol{u}}\right)\leq h\left(\boldsymbol{u}\right) \leq h\left(\bar{\boldsymbol{u}}\right)+a \right\}\]
\end{lemma}
The semialgebraic functions are the most frequently used functions with KL property. This class of functions has been used in the previous studies\cite{Ref_bolte2007lojasiewicz,Ref_bolte2007clarke,Ref_jiang2021proximal}. As we known, a real symmetric matrix is positive semidefinite if and only if all its eigenvalue are non-negative, which means that the $\mathcal{S}_+^d$ can be written as the intersection of finite many polynomial inequalities. Thus, the  $\mathcal{S}_+^d$ is a semialgebraic set because the semialgebraic property is stable under the boolean operations. Moreover, From the Tarski-Seidenberg theorem \cite[Theorem 8.6.6]{Ref_1993Algorithmic}, we know that the semidefinite programming representable sets are all semialgebraic\cite{Ref_ben2001lectures,Ref_ioffe2009invitation}.\par
Now, we are ready to present our global convergence of $\left\{\mathbf{U}^k\right\}$ generated by s-iPDCA. Similar to  the method proposed by Liu et al.\cite{Ref_liu2019refined}, we make use of the following auxiliary function:
\begin{equation} \label{eq_54}
E\left(\mathbf{U},\mathbf{W},\mathbf{V}, \mathbf{T}\right) = f_1\left(\mathbf{U}\right)-\langle \mathbf{U},\mathbf{W}\rangle + f^*_2\left(\mathbf{W}\right) +\frac{\alpha}{2}\|\mathbf{U}-\mathbf{V}\|^2_F-\langle \mathbf{T}, \mathbf{U}-\mathbf{V}\rangle.
\end{equation}
In here, $f^*_2$ is convex conjugate of $f_2$, given as
\begin{equation} \label{eq_55}
f^*_2\left(\mathbf{W}\right) = \sup_{\mathbf{U}\in\mathcal{S}_+^d} \left\{\langle \mathbf{W},\mathbf{U}\rangle-f_2\left(\mathbf{U}\right)\right\}.
\end{equation}
Then $f_1\left(\mathbf{U}\right)-f_2\left(\mathbf{U}\right)\leq f_1\left(\mathbf{U}\right)  -\langle \mathbf{U},\mathbf{W}\rangle + f^*_2\left(\mathbf{W}\right)$ holds true. 
Thanks to the fact that $f_2\left(\mathbf{U}\right) = c\|\mathbf{U}\|_{(r)}$ is a proper closed convex function, we have that $f^*_2\left(\mathbf{W}\right)$ is also a proper closed convex function and the Young’s inequality holds
\begin{equation} \label{eq_56}
f^*_2\left(\mathbf{W}\right) + f_2\left(\mathbf{U}\right)\geq \langle \mathbf{U},\mathbf{W}\rangle.
\end{equation}
and the equality holds if and only if $\mathbf{W}\in \partial f_2\left(\mathbf{U}\right)$. Moreover, for any $\mathbf{U}$ and $\mathbf{W}$,  $\mathbf{W}\in \partial f_2\left(\mathbf{U}\right)$  if and only if $\mathbf{U}\in \partial f^*_2\left(\mathbf{W}\right)$. Obviously, $E$ is semialgebraic function, so it is also KL function.\par
Based on the assumption that infinite serious steps are performed in s-iPDCA for solving \eqref{eq_11}, we know that  only finite stability centers generated in null step. The sequence $\left\{\mathbf{U}^k\right\}$ is  shown as 
 \begin{equation}\nonumber
\left\{\cdots,\underbrace{\mathbf{U}^{k-M}}_{\mathbf{U}^{k_{l}}},\underbrace{\mathbf{U}^{k-M+1},\cdots,\mathbf{U}^{k}}_{\text{M null steps}},\underbrace{\mathbf{U}^{k+1}}_{\mathbf{U}^{k_{l+1}}},\underbrace{\mathbf{U}^{k+2},\cdots,\mathbf{U}^{k+N+1}}_{{\text{N null steps}}},\underbrace{\mathbf{U}^{k+N+2}}_{\mathbf{U}^{k_{l+2}}},\cdots\right\},
\end{equation}
 in which, $\mathbf{U}^{k_l}$ denotes the stablity center generated in the $l^{th}$ serious step.  Since the subsequence \[\left\{\mathbf{U}^{k-M+1},\cdots, \mathbf{U}^{k},\mathbf{U}^{k+2},\cdots,\mathbf{U}^{k+N+1}\right\}\]  is the collection of  the  stablity centers in null steps between $l^{th}$ serious step and $(l+2)^{th}$ serious step, then $\mathbf{U}^{k_{l}}=\mathbf{U}^{k-M} = \mathbf{U}^{k-M+1}=\cdots=\mathbf{U}^{k}$ and $\mathbf{U}^{k_{l+1}} = \mathbf{U}^{k+1} = \mathbf{U}^{k+2}=\cdots=\mathbf{U}^{k+N+1}$ hold, which shows that the stablity centers in null steps are the finite repetition of that in serious steps. By removing the $\mathbf{U}^k$ generated in null steps from $\left\{\mathbf{U}^k\right\}$, we obtain a subsequence $\left\{\mathbf{U}^{k_l}\right\}$. In addition,  subsequence $\left\{\mathbf{W}^{k_{l}}\right\}$ denotes the set of   chosen subgradient of $f_2$ at $\mathbf{U}^{k_{l}}$. The subsequence $\left\{\mathbf{\Delta}^{k_{l}}\right\}$ is the set of inexact term related to $\mathbf{U}^{k_{l}}$, then $\|\mathbf{\Delta}^{k_{l}}\|_F\leq \frac{(1-\kappa)\alpha}{2}\|\mathbf{U}^{k_{l}}-\mathbf{U}^{k_{l-1}}\|_F$ holds.
\begin{proposition}\label{prop_6}
 Let E be defined in \eqref{eq_54}. Suppose that  infinite serious steps are performed in s-iPDCA for solving \eqref{eq_11}. Let $\left\{\mathbf{U}^{k_l}\right\}$, $\left\{\mathbf{\Delta}^{k_{l}}\right\}$and $\left\{\mathbf{W}^{k_l}\right\}$ be the subsequences generated in serious steps of s-iPDCA for solving \eqref{eq_11}.  Then the
following statements hold:\\
$(1)$ For any $l$,
 \begin{equation} \label{eq_57}
J_c\left(\mathbf{U}^{k_{l+1}}\right)
\leq E\left(\mathbf{U}^{k_{l+1}},\mathbf{W}^{k_l},\mathbf{U}^{k_l},\mathbf{\Delta}^{k_{l+1}}\right).
\end{equation}
$(2)$ For any $l$,
 \begin{equation} \label{eq_58}
E\left(\mathbf{U}^{k_{l+1}},\mathbf{W}^{k_l},\mathbf{U}^{k_l},\mathbf{\Delta}^{k_{l+1}}\right)-E\left(\mathbf{U}^{k_l},\mathbf{W}^{k_{l-1}},\mathbf{U}^{k_{l-1}},\mathbf{\Delta}^{k_{l}}\right)\leq \frac{\kappa\alpha}{2}\|\mathbf{U}^{k_l} -\mathbf{U}^{k_{l-1}}\|^2_F.
\end{equation}
$(3)$ The set of accumulation points of the sequence $\left\{\left(\mathbf{U}^{k_{l+1}}, \mathbf{W}^{k_l}, \mathbf{U}^{k_l},\mathbf{\Delta}^{k_{l+1}}\right)\right\}$, denoted by $\mathbf{\Gamma}$, is a nonempty compact set.
\\
$(4)$ The limit $\Upsilon = \lim_{l\rightarrow\infty}E\left(\mathbf{U}^{k_{l+1}},\mathbf{W}^{k_l},\mathbf{U}^{k_l},\mathbf{\Delta}^{k_{l+1}}\right)$ exists and $E \equiv \Upsilon$ on $\mathbf{\Gamma}$.
\\
$(5)$ There exists $\rho > 0$ such that for any $l \geq 1$, we have
 \begin{equation} \label{eq_59}
\operatorname{dist}\left(0,\partial E\left(\mathbf{U}^{k_{l+1}},\mathbf{W}^{k_l},\mathbf{U}^{k_l},\mathbf{\Delta}^{k_{l+1}}\right)\right)\leq \rho\|\mathbf{U}^{k_{l+1}}-\mathbf{U}^{k_l}\|_F.
\end{equation}
\end{proposition}
\begin{proof}
We first prove (1). 
Since $\mathbf{U}^{k_{l+1}} = \mathbf{V}^{k+1}$ is the  stability center generated in serious step,  then the test \eqref{eq_15} holds true, shown as
\begin{equation}\label{60}
\|\mathbf{\Delta}^{k_{l+1}}\|_F\leq\frac{\left(1-\kappa\right)\alpha}{2}\|\mathbf{U}^{k_{l+1}}-\mathbf{U}^{k_l}\|_F,
\end{equation}
then we get
\begin{equation}\label{eq_61}
\frac{\kappa\alpha}{2}\|\mathbf{U}^{k_{l+1}}-\mathbf{U}^{k_l}\|_F^2\leq \frac{\alpha}{2}\|\mathbf{U}^{k_{l+1}}-\mathbf{U}^{k_l}\|_F^2-\langle\mathbf{\Delta}^{k_{l+1}},\mathbf{U}^{k_{l+1}}-\mathbf{U}^{k_l}\rangle.
\end{equation}
Thus, we have
\begin{equation}\nonumber 
\begin{array}{lll}
&&E\left(\mathbf{U}^{k_{l+1}},\mathbf{W}^{k_l},\mathbf{U}^{k_l},\mathbf{\Delta}^{k_{l+1}}\right)\\
&=&  f_1\left(\mathbf{U}^{k_{l+1}}\right)-\langle \mathbf{U}^{k_{l+1}},\mathbf{W}^{k_l}\rangle + f^*_2\left(\mathbf{W}^{k_l}\right)+\frac{\alpha}{2}\|\mathbf{U}^{k_{l+1}} -\mathbf{U}^{k_l}\|^2_F-\langle\mathbf{\Delta}^{k_{l+1}},\mathbf{U}^{k_{l+1}}-\mathbf{U}^{k_l}\rangle\\
&=&  f_1\left(\mathbf{U}^{k_{l+1}}\right)-\langle \mathbf{U}^{k_{l+1}}-\mathbf{U}^{k_l},\mathbf{W}^{k_l}\rangle - f_2\left(\mathbf{U}^{k_l}\right)+\frac{\alpha}{2}\|\mathbf{U}^{k_{l+1}} -\mathbf{U}^{k_l}\|^2_F-\langle\mathbf{\Delta}^{k_{l+1}},\mathbf{U}^{k_{l+1}}-\mathbf{U}^{k_l}\rangle\\
&\geq&  f_1\left(\mathbf{U}^{k_{l+1}}\right)-  f_2\left(\mathbf{U}^{k_{l+1}}\right)+\frac{\alpha}{2}\|\mathbf{U}^{k_{l+1}} -\mathbf{U}^{k_l}\|^2_F-\langle\mathbf{\Delta}^{k_{l+1}},\mathbf{U}^{k_{l+1}}-\mathbf{U}^{k_l}\rangle\\
&\geq&  f_1\left(\mathbf{U}^{k_{l+1}}\right)-  f_2\left(\mathbf{U}^{k_{l+1}}\right)+\frac{\kappa\alpha}{2}\|\mathbf{U}^{k_{l+1}} -\mathbf{U}^{k_l}\|^2_F\\
&\geq&  f_1\left(\mathbf{U}^{k_{l+1}}\right)-  f_2\left(\mathbf{U}^{k_{l+1}}\right) = J_c\left(\mathbf{U}^{k_{l+1}}\right).
 \end{array}
\end{equation}
In here, the second euqality follows from  the convexity of $f_2$ and the fact that $\mathbf{W}^{k_l}\in\partial f_2\left(\mathbf{U}^{k_l}\right)$, the first ineuqality follows from  the convexity of $f_2$.\par
For state (2). Since $\mathbf{U}^{k_{l+1}} =\mathbf{V}^{k+1}\in\mathcal{S}_+^d$ is the optimal solution of strongly convex problem \eqref{eq_16}, then the following inequality follow from the feasibility of $\mathbf{U}^{k_l}$:
\begin{equation}\label{eq_62}
\begin{array}{lll}
& f_1(\mathbf{U}^{k_{l+1}})-\langle \mathbf{U}^{k_{l+1}},\mathbf{W}^{k_l}\rangle+\frac{\alpha}{2}\|\mathbf{U}^{k_{l+1}}-\mathbf{U}^{k_l}\|_F^2-\langle\mathbf{\Delta}^{k_{l+1}},\mathbf{U}^{k_{l+1}}\rangle\\
&\leq f_1(\mathbf{U}^{k_l})-\langle \mathbf{U}^{k_l},\mathbf{W}^{k_l}\rangle -\langle\mathbf{\Delta}^{k_{l+1}},\mathbf{U}^{k_l}\rangle.
\end{array}
\end{equation}
then we know that
\begin{equation} \nonumber
\begin{array}{lll}
& &E\left(\mathbf{U}^{k_{l+1}},\mathbf{W}^{k_l},\mathbf{U}^{k_l},\mathbf{\Delta}^{k_{l+1}}\right)\\
&=&  f_1\left(\mathbf{U}^{k_{l+1}}\right)-\langle \mathbf{U}^{k_{l+1}},\mathbf{W}^{k_l}\rangle + f^*_2\left(\mathbf{W}^{k_l}\right)+\frac{\alpha}{2}\|\mathbf{U}^{k_{l+1}} -\mathbf{U}^{k_l}\|^2_F
-\langle\mathbf{\Delta}^{k_{l+1}},\mathbf{U}^{k_{l+1}}-\mathbf{U}^{k_l}\rangle\\
&\leq& f_1(\mathbf{U}^{k_l})+\langle \mathbf{U}^{k_{l+1}}-\mathbf{U}^{k_l},\mathbf{W}^{k_l}\rangle-\frac{\alpha}{2}\|\mathbf{U}^{k_{l+1}}-\mathbf{U}^{k_l}\|_F^2 +\langle\mathbf{\Delta}^{k_{l+1}},\mathbf{U}^{k_{l+1}}-\mathbf{U}^{k_l}\rangle\\
&&-\langle \mathbf{U}^{k_{l+1}},\mathbf{W}^{k_l}\rangle 
+ f^*_2\left(\mathbf{W}^{k_l}\right)+\frac{\alpha}{2}\|\mathbf{U}^{k_{l+1}} -\mathbf{U}^{k_l}\|^2_F-\langle\mathbf{\Delta}^{k_{l+1}},\mathbf{U}^{k_{l+1}}-\mathbf{U}^{k_l}\rangle\\
& =&  f_1(\mathbf{U}^{k_l})-\langle \mathbf{U}^{k_l},\mathbf{W}^{k_l}\rangle + f^*_2\left(\mathbf{W}^{k_l}\right).
 \end{array}
\end{equation}
Simlar to  inequality \eqref{eq_61}, the following inequality:
\begin{equation}\label{eq_63}
\frac{\kappa\alpha}{2}\|\mathbf{U}^{k_l}-\mathbf{U}^{k_{l-1}}\|_F^2\leq \frac{\alpha}{2}\|\mathbf{U}^{k_l}-\mathbf{U}^{k_{l-1}}\|_F^2-\langle\mathbf{\Delta}^{k_{l}},\mathbf{U}^{k_l}-\mathbf{U}^{k_{l-1}}\rangle.
\end{equation}
holds true.
Consequently, we have
\begin{equation}\nonumber 
\begin{array}{lll}
&&E\left(\mathbf{U}^{k_{l+1}},\mathbf{W}^{k_l},\mathbf{U}^{k_l},\mathbf{\Delta}^{k_{l+1}}\right)\\
& \leq & f_1(\mathbf{U}^{k_l})-\langle \mathbf{U}^{k_l},\mathbf{W}^{k_l}\rangle + f^*_2\left(\mathbf{W}^{k_l}\right)=  f_1(\mathbf{U}^{k_l}) - f_2\left(\mathbf{U}^{k_l}\right)\\
& \leq&  f_1(\mathbf{U}^{k_l})-\langle \mathbf{U}^{k_l},\mathbf{W}^{k_{l-1}}\rangle + f^*_2\left(\mathbf{W}^{k_{l-1}}\right)\\
& =& E\left(\mathbf{U}^{k_l},\mathbf{W}^{k_{l-1}},\mathbf{U}^{k_{l-1}},\mathbf{\Delta}^{k_{l}}\right)- \frac{\alpha}{2}\left\|\mathbf{U}^{k_l} -\mathbf{U}^{k_{l-1}}\right\|^2_F+ \langle\mathbf{\Delta}^{k_{l}},\mathbf{U}^{k_l}-\mathbf{U}^{k_{l-1}}\rangle\\
&\leq& E\left(\mathbf{U}^{k_l},\mathbf{W}^{k_{l-1}},\mathbf{U}^{k_{l-1}},\mathbf{\Delta}^{k_{l-1}}\right)- \frac{\kappa\alpha}{2}\|\left(\mathbf{U}^{k_l} -\mathbf{U}^{k_{l-1}}\right)\|^2_F.
 \end{array}
\end{equation}
In here, the first equality follows from the convexity of $f_2$ and the fact that $\mathbf{W}^{k_l}\in\partial f_2\left(\mathbf{U}^{k_l}\right)$. The second inequality follows from the convexity of $f_2$ and the Young’s inequality applied to $f_2$. The last inequality come from \eqref{eq_63}.\par
For space limitations, we omitted the proof of statements (3)-(5),  and we refer the intrested readers to \cite{Ref_wen2018proximal,Ref_liu2019refined}. This completes the proof.
\end{proof}
\begin{theorem}\label{thm:2}
Set the tolerance error $\varepsilon = 0$.  Let $\{\mathbf{U}^{k_l}\}$ be the  stablity center sequence generated in serious steps of s-iPDCA for solving \eqref{eq_11}. Then  $\{\mathbf{U}^{k_l}\}$ is convergent to a stationary point of \eqref{eq_11}. Moreover, $\mathbf{\Sigma}_{l=0}^{\infty}\|\mathbf{U}^{k_{l+1}}-\mathbf{U}^{k_{l}}\|_F<\infty$.
\end{theorem}
\begin{proof}
From Proposition \ref{prop_6}, we know that  $\left\{E\left(\mathbf{U}^{k_{l+1}},\mathbf{W}^{k_l},\mathbf{U}^{k_l},\mathbf{\Delta}^{k_{l+1}}\right) \right\}$ is nonincreasing and its limitation $\Upsilon$ exists.  We first show that $E\left(\mathbf{U}^{k_{l+1}},\mathbf{W}^{k_l},\mathbf{U}^{k_l},\mathbf{\Delta}^{k_{l+1}}\right) > \Upsilon,  \forall l>0$. To this end, we suppose that  $\exists L>0$ such that $E\left(\mathbf{U}^{k_{L+1}}, \mathbf{W}^{k_L}, \mathbf{U}^{k_L},\mathbf{\Delta}^{k_{L+1}}\right) = \Upsilon$, then $E\left(\mathbf{U}^{k_{l+1}},\mathbf{W}^{k_l},\mathbf{U}^{k_l},\mathbf{\Delta}^{k_{l+1}}\right) = \Upsilon$ holds ture for all $l>L$. From \eqref{eq_58}, we have $\mathbf{U}^{k_l} =\mathbf{U}^{k_L}$, $\forall l \geq L$. This shows that only finite serious steps are performed for s-iPDCA, which is contrary to the assumption of infinite serious steps.\par
 Next, from the Theorem \ref{thm:1}, we know that it is sufficient to show the convergence of $\{\mathbf{U}^{k_l}\}$  and $\mathbf{\Sigma}_{l=0}^{\infty}\|\mathbf{U}^{k_{l+1}}-\mathbf{U}^{k_l}\|_F$.  Since $E$ satisfies the KL property at each point in the compact set $\mathbf{\Gamma}\subset\operatorname{dom} \partial E$ and $E\equiv\Upsilon$ on $\mathbf{\Gamma}$, by Lemma \ref{lemma_1}, there exists $\epsilon>0$  and a continuous concave function $\varphi\in\Xi_a$ with $a > 0$ such that
\begin{equation}\nonumber
\varphi^{\prime}\left(E\left(\mathbf{U},\mathbf{W},\mathbf{V}, \mathbf{T}\right)-\Upsilon\right) \cdot \operatorname{dist}(\mathbf{0}, \partial E\left(\mathbf{U},\mathbf{W},\mathbf{V},\mathbf{T}\right) \geq 1.
\end{equation}
$\forall \left(\mathbf{U},\mathbf{W},\mathbf{V},\mathbf{T}\right)\in \mathbf{\Theta}$, with
\begin{equation}\nonumber
\begin{array}{ll}
\mathbf{\Theta}=&\left\{\left(\mathbf{U},\mathbf{W},\mathbf{V},\mathbf{T}\right): \operatorname{dist}(\left(\mathbf{U},\mathbf{W},\mathbf{V},\mathbf{T}\right), \mathbf{\Gamma})<\epsilon\right\} \\
&\cap\left\{\left(\mathbf{U},\mathbf{W},\mathbf{V},\mathbf{T}\right): \Upsilon<E\left(\mathbf{U},\mathbf{W},\mathbf{V},\mathbf{T}\right)<\Upsilon+a\right\}
\end{array}
\end{equation}
Since $\mathbf{\Gamma}$ is the set of accumulation points of the  $\left\{\left(\mathbf{U}^{k_{l+1}}, \mathbf{W}^{k_l}, \mathbf{U}^{k_l},\mathbf{\Delta}^{k_{l+1}}\right)\right\}$, we know that
\begin{equation}\nonumber
\lim_{l\rightarrow\infty}\operatorname{dist} \left(\left(\mathbf{U}^{k_{l+1}}, \mathbf{W}^{k_l}, \mathbf{U}^{k_l},\mathbf{\Delta}^{k_{l+1}}\right),\mathbf{\Gamma}\right) = 0.
\end{equation}
Thus, there exists  $\bar{L}>0$ such that 
\[\operatorname{dist} \left(\left(\mathbf{U}^{k_{l+1}}, \mathbf{W}^{k_l}, \mathbf{U}^{k_l},\mathbf{\Delta}^{k_{l+1}}\right),\mathbf{\Gamma}\right)<\epsilon, \forall l>\bar{L}-2.\]
 From Proposition \ref{prop_6}, we know that the sequence $\left\{E\left(\mathbf{U}^{k_{l+1}},\mathbf{W}^{k_l},\mathbf{U}^{k_l},\mathbf{\Delta}^{k_{l+1}}\right) \right\}$ converges to $\Upsilon$, then exists $\bar{\bar{L}}>0$ such that 
 \[\Upsilon<E\left(\mathbf{U}^{k_{l+1}},\mathbf{W}^{k_l},\mathbf{U}^{k_l},\mathbf{\Delta}^{k_{l+1}}\right)<\Upsilon+a, \forall l>\bar{\bar{L}}-2.\]
  Let $\tilde{L} = \max\left\{\bar{L},\bar{\bar{L}}\right\}$ and \[E^{k_{l-1}} = E\left(\mathbf{U}^{k_{l-1}},\mathbf{W}^{k_{l-2}},\mathbf{U}^{k_{l-2}},\mathbf{\Delta}^{k_{l-1}}\right).\] Therefore,  $\forall l>\tilde{L}$, we have that $\left(\mathbf{U}^{k_{l-1}}, \mathbf{W}^{k_{l-2}}, \mathbf{U}^{k_{l-2}},\mathbf{\Delta}^{k_{l-1}} \right)\in\mathbf{\Theta}$  and 
\begin{equation}\label{eq_64}
\varphi^{\prime}\left(E^{k_{l-1}}-\Upsilon\right) \cdot \operatorname{dist}\left(\mathbf{0}, \partial E^{k_{l-1}} \right)\geq 1
\end{equation}
hold true. Using the concavity of $\varphi$, we know that 
\begin{equation}\label{eq_65}
\begin{array}{ll}
&\left[\varphi\left(E^{k_{l-1}}-\Upsilon\right) -\varphi\left(E^{k_{l+1}}-\Upsilon\right)\right] \cdot \operatorname{dist}(\mathbf{0}, \partial E^{k_{l-1}}) \\
&\geq\varphi^{\prime}\left(E^{k_{l-1}}-\Upsilon\right)\cdot \operatorname{dist}\left(\mathbf{0}, \partial E^{k_{l-1}}\right)\cdot \left(E^{k_{l-1}}-E^{k_{l+1}}\right) \geq E^{k_{l-1}}-E^{k_{l+1}}
\end{array}
\end{equation}
holds true,  $\forall l>\tilde{L}$. In here, the last inequality follows from \eqref{eq_64}.
Let $\mathbf{\pi}^{k_{l-1}} =  \varphi\left(E^{k_{l-1}}-\Upsilon\right)$ and $\mathbf{\pi}^{k_{l+1}} =  \varphi\left(E^{k_{l+1}}-\Upsilon\right)$. Combining the results in \eqref{eq_58}, \eqref{eq_59} and \eqref{eq_65}, we have 
\begin{equation}\label{eq_66}
\begin{array}{ll}
&\left\|\mathbf{U}^{k_l} -\mathbf{U}^{k_{l-1}}\right\|^2_F+\left\|\mathbf{U}^{k_{l-1}} -\mathbf{U}^{k_{l-2}}\right\|^2_F\leq \frac{2\rho}{\kappa\alpha}\left(\mathbf{\pi}^{k_{l-1}} -\mathbf{\pi}^{k_{l+1}}\right)\left\|\mathbf{U}^{k_{l-1}} -\mathbf{U}^{k_{l-2}}\right\|_F
\end{array}
\end{equation}
Applying the arithmetic mean–geometric mean inequality, we obtain
\begin{equation}\nonumber
\begin{array}{ll}
\left\|\mathbf{U}^{k_l} -\mathbf{U}^{k_{l-1}}\right\|_F&\leq \sqrt{\frac{\rho}{\kappa\alpha}\left(\mathbf{\pi}^{k_{l-1}} -\mathbf{\pi}^{k_{l+1}}\right)-\frac{1}{2}\left\|\mathbf{U}^{k_{l-1}} -\mathbf{U}^{k_{l-2}}\right\|_F}\cdot\sqrt{2\left\|\mathbf{U}^{k_{l-1}} -\mathbf{U}^{k_{l-2}}\right\|_F}\\
&\leq \frac{\rho}{2\kappa\alpha}\left(\mathbf{\pi}^{k_{l-1}} -\mathbf{\pi}^{k_{l+1}}\right)-\frac{1}{4}\left\|\mathbf{U}^{k_{l-1}} -\mathbf{U}^{k_{l-2}}\right\|_F+ \left\|\mathbf{U}^{k_{l-1}} -\mathbf{U}^{k_{l-2}}\right\|_F.
\end{array}
\end{equation}
Then, we have
\begin{equation}\label{eq_67}
\frac{1}{4}\left\|\mathbf{U}^{k_l} -\mathbf{U}^{k_{l-1}}\right\|_F\leq \frac{\rho}{2\kappa\alpha}\left(\mathbf{\pi}^{k_{l-1}} -\mathbf{\pi}^{k_{l+1}}\right)+ \frac{3}{4}\left(\left\|\mathbf{U}^{k_{l-1}} -\mathbf{U}^{k_{l-2}}\right\|_F-\left\|\mathbf{U}^{k_l} -\mathbf{U}^{k_{l-1}}\right\|_F\right)
\end{equation}
Summing both sides of \eqref{eq_67} from $l=\tilde{L}$ to $\infty$, we have
\begin{equation}\label{eq_68}
\begin{array}{lll}
\frac{1}{4}\Sigma_{l=\tilde{L}}^{\infty}\left\|\mathbf{U}^{k_l} -\mathbf{U}^{k_{l-1}}\right\|_F&\leq& \frac{\rho}{2\kappa\alpha}\left(\mathbf{\pi}^{k_{\tilde{L}-1}} +\mathbf{\pi}^{k_{\tilde{L}}}\right)- \lim_{l\rightarrow\infty}\frac{\rho}{2\kappa\alpha}\left(\mathbf{\pi}^{k_{l}} +\mathbf{\pi}^{k_{l+1}}\right)\\
&&+ \frac{3}{4}\left(\left\|\mathbf{U}^{k_{\tilde{L}-1}} -\mathbf{U}^{k_{\tilde{L}-2}}\right\|_F-\lim_{l\rightarrow\infty}\left\|\mathbf{U}^{k_l} -\mathbf{U}^{k_{l-1}}\right\|_F\right)
\end{array}
\end{equation}
By applying the fact that $\lim_{l\rightarrow\infty}\frac{\rho}{2\kappa\alpha}\left(\mathbf{\pi}^{k_{l}} +\mathbf{\pi}^{k_{l+1}}\right)=0$ and $\lim_{l\rightarrow\infty}\left\|\mathbf{U}^{k_l} -\mathbf{U}^{k_{l-1}}\right\|_F=0$, we obtain
\begin{equation}\label{eq_69}
\begin{array}{lll}
\frac{1}{4}\Sigma_{l=\tilde{L}}^{\infty}\left\|\mathbf{U}^{k_l} -\mathbf{U}^{k_{l-1}}\right\|_F\leq \frac{\rho}{2\kappa\alpha}\left(\mathbf{\pi}^{k_{\tilde{L}-1}} +\mathbf{\pi}^{k_{\tilde{L}}}\right)+ \frac{3}{4}\left\|\mathbf{U}^{k_{\tilde{L}-1}} -\mathbf{U}^{k_{\tilde{L}-2}}\right\|_F<\infty
\end{array}
\end{equation}
Thus the subsequence $\left\{\mathbf{U}^{k_l}\right\}$ is convergent as well as $\mathbf{\Sigma}_{l=0}^{\infty}\|\mathbf{U}^{k_{l+1}}-\mathbf{U}^{k_l}\|_F<\infty$. Combining this with the results of Theorem \ref{thm:1}, we know that the sequence $\left\{\mathbf{U}^{k_l}\right\}$ generated by s-iPDCA converges to  the stationary points of \eqref{eq_11}. This completes the proof.
\end{proof}
\begin{theorem}\label{thm:3}
Set the tolerance error $\varepsilon = 0$.  Let $\{\mathbf{U}^{k}\}$ be the  stablity center sequence generated by s-iPDCA for \eqref{eq_11}. Then  $\{\mathbf{U}^{k}\}$ is convergent to a stationary point of \eqref{eq_11}. Moreover, $\mathbf{\Sigma}_{k=0}^{\infty}\|\mathbf{U}^{k+1}-\mathbf{U}^{k}\|_F<\infty$.
\end{theorem}
\begin{proof}
From assumption of infinite serious steps in this subsection, we know that $\{\mathbf{U}^{k_{l}}\}$ is just the subsequence of $\{\mathbf{U}^{k}\}$ removing the finite repeated points. From the Teorem \ref{thm:2}, we know that the sequence $\left\{\mathbf{U}^{k}\right\}$ is also convergent to the stationary points of DC problem \eqref{eq_11}. Moreover, we have \[\mathbf{\Sigma}_{k=0}^{\infty}\|\mathbf{U}^{k+1}-\mathbf{U}^{k}\|_F = \mathbf{\Sigma}_{l=0}^{\infty}\|\mathbf{U}^{k_{l+1}}-\mathbf{U}^{k_l}\|_F<\infty.\]  This completes the proof.
\end{proof}
\section{Numerical experiments}\label{sec:5}
In this section, we  perform numerical experiments to show the efficiency of our s-iPDCA for
solving the RCLSSDP \eqref{eq_7}, and then to prove the effectiveness of the RCKSDPP in face recognition.  All experiments are performed in Matlab 2020a on a 64-bit PC with an Intel(R) Xeon(R)  CPU E5-2609 v2 (2.50GHz)(2 processor) and 56GB of RAM.\par
Before we start the experiments, we first scale the model in \eqref{eq_6} as following: let $\tilde{\mathcal{A}}= \frac{\mathcal{A}}{\|\mathbf{A}\|_F}$ and $\tilde{\boldsymbol{b}}= \frac{\boldsymbol{b}}{\|\boldsymbol{b}\|}$, then the scaled RCLSSDP shown as 
 \begin{equation}\label{eq_70}
\begin{array}{ll}
\min_{\tilde{\mathbf{U}}\in \mathcal{S}^d_+} &J(\tilde{\mathbf{U}}) = \|\tilde{\mathcal{A}}\left(\tilde{\mathbf{U}}\right)- \tilde{\boldsymbol{b}}\|^{2}\\
\text { s.t. }&\langle \tilde{\mathbf{U}},\mathbf{I}\rangle-\|\tilde{\mathbf{U}}\|_{(r)} =  0.
\end{array}
\end{equation}
In here, $\tilde{\mathbf{U}} = \frac{\|\mathbf{A}\|_F\mathbf{U}}{\|\boldsymbol{b}\|}$. Then, we use the s-iPDCA or classical PDCA to solve the scaled DC problem from above scaled RCLSSDP. After the scaled optimal solution and optimal value being computed, the optimal solution and optimal value are obtaind by rescaling. For convenience, we omit the '$\sim$' in anywhere else of this paper. Next, we divide our numerical experiments  into two parts, shown in the following. In the first part, we  compare the performance of s-iPDCA with the classical PDCA  and the PDCA with extrapolation (PDCAe) on solving the  RCLSSDP \eqref{eq_7}, where the RCLSSDP come from dimension reduction for the COIL-20 database. In the second part, we apply RCKSDPP on dimension reduction for face recognition on the ORL database and the YaleB database, the results is compared with KSSDPP and KPCA.\par
The  dimension  of the images in some databases is mach larger than the number  of data points, e.g., the ORL database contains 400 face images with  $d = 112\times 92$, cropped YaleB database contains 2414 face images with $d = 168\times 192$ and the cropped COIL-20 database contains 1440 images with $d= 128\times128$. In these situations, we use kernel trick to reduce the size of PSD matrix in \eqref{eq_11} from $d \times d$ to $n\times n$. The kernel extension SDPP (KSDPP) first  maps the data from the original input space to another higher (even infinite) dimensional feature space $\phi: \mathcal{X}\rightarrow \mathcal{H}$.  Then we can rewrite the projection matrix as $\tilde{\mathbf{P}} = \mathbf{\Psi}\mathbf{P}$ with $\mathbf{\Psi} = \left[\phi(\boldsymbol{x}_1),...,\phi(\boldsymbol{x}_n)\right]$. To obtain the KSDPP, we only need to replace the projected matrix $\mathbf{P}$ and  the input covariate $\boldsymbol{x}$ in \eqref{eq_1} as $\tilde{\mathbf{P}}$ and $\phi(\boldsymbol{x})$, respectively. Denote the element of the kernel matrix $\mathbf{K}$ as $\mathbf{K}_{i j}=k\left(\mathbf{x}_{i}, \mathbf{x}_{j}\right)=\left\langle\phi\left(\mathbf{x}_{i}\right), \phi\left(\mathbf{x}_{j}\right)\right\rangle$, then we get the KSDPP
\begin{equation}\label{eq_71}
\begin{array}{ll}
\min & J(\mathbf{P})=\frac{1}{n} \sum_{i=1}^{n} \sum_{\mathbf{K}_{j} \in \mathcal{N}\left(\mathbf{K}_{i}\right)}\left( \|\mathbf{P}^T\mathbf{K}_{i} - \mathbf{P}^T\mathbf{K}_{j}\|^2-\left\|\boldsymbol{y}_{i}-\boldsymbol{y}_{j}\right\|^{2}\right)^{2}. 
\end{array}
\end{equation}
Simlar to the SDPP, the KSDPP also can be equivalently transferred into  rank constrianed KSDPP (RCKSDPP). What's more, we call the kernel extension of SSDPP  as KSSDPP.
In this paper, the Gaussian kernel function $exp(\frac{-\|\boldsymbol{x}-\boldsymbol{y}\|^2}{\varsigma t^2})$ is used in this paper, in here, the $t$ and $\varsigma>0$ are the kernel parameters. The $t$ is set by Silverman rule of thumb, shown as $t  = 1.06n^{-0.2}\sqrt{\frac{\Sigma_{i=1}^n\|\boldsymbol{x}_i-\bar{\boldsymbol{x}}\|^2}{n-1}}$. The parameter $\varsigma$ is set as $\varsigma = 2$ in stadrad Gaussian kernel function, but in this paper, we adjust the $\varsigma$ to achieve better numerical performance. 

\par
In this paper, we use the low precision solution of convex problem \eqref{eq_4} as the initial solution of our s-iPDCA. When a sufficiently small  penalty parameter is chosen, the \eqref{eq_8} can be solved easily with the initial solution from convex problem \eqref{eq_4}. By this observation, we start our s-iPDCA with very small penalty parameter.\par
\subsection{Comparing the performance of s-iPDCA with the classial PDCA and the PDCAe}\label{sec:5_1}
In this subsection, we  compare the performance of our s-iPDCA  with the classical PDCA and the  PDCAe for solving the RCLSSDP \eqref{eq_7}. The RCLSSDP in this experiments comes  from performing DR on COIL-20 database. The projection dimension is set as $r = 5$ and the kernel parameter is set as $\varsigma = 2$. The neighborhood size of k-nearest neighbor (k-nn) for the RCKSDPP is set as $k = \operatorname{round}(log(n))$. The proximal parameters of the s-iPDCA and PDCA are both set as $\alpha = 5\times10^{-6}$.  In order to obtain the same quality of suboptimal objective function value $J$, we set the  termination criteria of these three algorithm as $\eta = \frac{|J(\mathbf{U}^{k+1})-J(\mathbf{U}^{k})|}{1+|J(\mathbf{U}^{k})|}\leq 1\times 10^{-7}$ for the PDCAe, $\eta = \frac{|J(\mathbf{U}^{k+1})-J(\mathbf{U}^{k})|}{1+|J(\mathbf{U}^{k})|}\leq 7\times 10^{-8}$ for the PDCA and $\eta = \frac{|J(\mathbf{V}^{k+1})-J(\mathbf{U}^{k})|}{1+|J(\mathbf{U}^{k})|}\leq 7\times 10^{-8}$ for the s-iPDCA, respectively. Then the total solving time (t/s), the outer iterations (Iter), the relative variation of objective function value ($\eta$) and optimal value of DC problem \eqref{eq_11} ($J$) were compared, the results are listed in Table \ref{tab:1}.\par
\begin{table}
\caption{Performance of  PDCAe, PDCA and s-iPDCA}
\label{tab:1}       
\begin{tabular}{cccccccccccccccc} 
\hline\\[2pt]  
    \centering{n}&\multicolumn{4}{c}{ PDCAe}&\multicolumn{4}{c}{PDCA}&\multicolumn{4}{c}{s-iPDCA}\\[2pt] 
    \cmidrule(lr){2-5} \cmidrule(lr){6-9} \cmidrule(lr){10-13}\\[2pt]
    &$J$&$\eta$&Iter&t/s&$J$&$\eta$&Iter&t/s&$J$&$\eta$&Iter&t/s\\[5pt] 
    \hline \\[2pt]
    100&3.89e-3&3.64e-8&391&2.51&4.06e-4&6.78e-8&153&5.68&2.90e-4&1.03e-9&137&1.74\\[10pt]
    200&7.94e-4&9.94e-8&468&10.09&2.51e-4&6.95e-8&204&15.84&2.42e-4&6.94e-8&200&6.22\\[10pt] 
   400&6.91e-3&9.95e-8&807&75.69&3.80e-3&6.98e-8&332&75.37&3.79e-3&6.96e-8&338&48.70\\[10pt]
   600&1.30e-2&9.95e-8&939&188.28&8.87e-3&6.99e-8&375&205.58&8.87e-3&6.97e-8&378&100.18\\[10pt]
   800&2.16e-2&9.96e-8&980&427.77&1.67e-2&6.99e-8&397&447.39&1.67e-2&6.98e-8&399&216.21\\[10pt]
   1000&2.60e-2&9.96e-8&961&713.44&2.60e-2&6.99e-8&503&739.89&2.58e-2&6.99e-8&506&435.15\\[10pt]
    1200&4.48e-2&9.98e-8&986&1199.66&4.11e-2&6.99e-8&506&1204.47&4.08e-2&6.99e-8&510&666.69\\[10pt]
     1440&4.83e-2&9.99e-8&1083&2128.23&4.45e-2&7.00e-8&530&1484.42&4.45e-2&6.99e-8&530&1057.91\\[5pt]
    \hline   
   \end{tabular}
\end{table}
When we use the PDCAe to solve the DC problem \eqref{eq_11}, the following convex subproblem is considered:
\begin{equation}\label{eq_72}
\min_{\mathbf{U} \in \mathcal{S}^n_+} \langle \frac{2}{n}\mathcal{A}(\mathcal{A}(\tilde{\mathbf{U}}^k)- \boldsymbol{b}), \mathbf{U}\rangle+c\langle \mathbf{U},\mathbf{I}\rangle- \langle \mathbf{U}, \mathbf{W}^{k}\rangle+\frac{L}{2}\|\mathbf{U}-\tilde{\mathbf{U}}^k\|_F,
\end{equation}
in here, $\tilde{\mathbf{U}}^k$ is the extrapolation term, given as $\tilde{\mathbf{U}}^k = \tilde{\mathbf{U}}^k + \beta_k(\mathbf{U}^k- \mathbf{U}^{k-1})$, the extrapolation parameter is set as $t_{k+1} = \frac{1+\sqrt{1+4t_k^2}}{2}$, $\beta_k = \frac{t_k-1}{t_{k+1}}$ with $t_0=1$. The $L$ is positive number greater than  the gradient Lipschitz smooth constant of $J(\mathbf{U})=\frac{1}{n}\|\mathcal{A}(\mathbf{U})-\boldsymbol{b}\|^2$, set as $L \geq \frac{2}{n}\|\mathbf{A}^{\top}\mathbf{A}\|_2$.  The \eqref{eq_72} has closed form solution, shown as
\[\mathbf{U}^{k+1} = \frac{1}{L}\mathbf{\Pi}_{\mathcal{S}^n_+}(L\tilde{\mathbf{U}}^{k}+\mathbf{W}^{k} -c\mathbf{I}-\frac{2}{n}\mathcal{A}^*(\mathcal{A}(\tilde{\mathbf{U}}^k)- \boldsymbol{b})).\]
 The algorithm details for PDCAe can be founded in \cite{Ref_liu2019refined,Ref_wen2018proximal}. In addition, the subproblem of PDCA is solved by the ABCD \ref{alg_abcd} and the termination error is set as $\zeta = 1\times10^{-9}$. When the s-iPDCA is used to solve the RCLSSDP \eqref{eq_11}, we set the initial inexact bounded $\epsilon_1 = 1\times 10^{-4}$. Then, the computation time (t/s), iteration (Iter) and optimal value ($J$) are compared with our s-iPDCA.\par
We perform dimension reduction  experiment on the COIL-20 database,  this data set  is a collection of gray images of 1440 images from 20 objects where each object has 72 different images. The  images of each objects, with uniformly distributed rotation angles, $\left[0^{\circ}, 5^{\circ}, \cdots,355^{\circ}\right]$, has been cropped into size of $128\times 128$. 
In order to compare the efficiency of our s-iPDCA with the PDCA and  PDCAe, we perform DR on the same images group, e.g., the first 10 images of the 20 objects have been chosen when n = 200 and the first 20 images of the all 20 objects have been chosen when n =400, etc. \par
 The results is listed in Table \ref{tab:1}, from which  we  know that our s-iPDCA outperforms the PDCA and the PDCAe  for solving the RCLSSDP \eqref{eq_7}  from  computation time and optimal value. Since only about $2-3$ steps the ABCD  have been performed at each s-iPDCA iteration, although the iteration  of the s-iPDCA is larger than the PDCA,  its total computation time is much less than PDCA. As is shown in Table \ref{tab:1}, the optimal value of s-iPDCA is smaller than that of the other two algorithms. What's more, although the subproblem of PDCAe has closed form solution, it takes the more computation time and iterations than the s-iPDCA to  solve the RCLSSDP for  all situation. The reason is that the proximal parameter $L$ of the PDCAe is chosen as the  gradient Lipschitz smooth constant of $J$ (or larger) to ensure the convergence of the PDCAe, which limits the convergent speed of the PDCAe.
\subsection{Dimension reduction for recognition}\label{sec:5_2}
In this subsection, some numerical experiments for face recogition are performed to show the advantage of our model in practical application. We divide face images database into training set and testing set randomly, then we proceed  the following  steps for face recognition. Firstly, a projected matrix  is obtained by solving the RCLSSDP \eqref{eq_6}. Then, we reduce the dimension of  face images in both training set and testing set  by applying the above projected martrix. Finally, we use the nearest neighbor method as classifier to identify whether the  projected covariate is belong which individual. \par
Two well-known face database ORL and the Extended Yale Face Database B (YaleB) \cite{Ref_georghiades2001few} are used in our experiments. The ORL database contains 400 images from 40 individuals where each individual has 10 different images. The size of the each image is $92\times 112$. For each individual, the face in the images are rotated, scaled or tilted to a mild degree. 
We only extract the subset of YaleB database that containing 2,414 frontal pose images of 38 individuals under different illuminations for ecah individual. We crop all the images from YaleB database  into $168\times 192$ pixels.\par
From each of the face database mentioned above, the image set is partitioned into the
different training and testing sets. We use the $G_p/T_q$ indicates that p face images per
individual are randomly selected for training and q face images from the remaining are used for testing. Next, we will show the advantage of the RCKSDPP by comparing its performance with KSSDPP and KPCA. The KSSDPP is solved by an SDP code, called boundary point method\cite{Ref_povh2006boundary}.\par

 According to our test, when we reduce the demension of  image data in ORL to around 40, the highest recognition accuracy is obtained. Then we set $r = 40$   when we use the RCKSDPP  to reduce the dimension of image data in the face recognition task on the ORL. In the DR experiment on the ORL, the neighborhood size of k-nn for the RCKSDPP is set as $k = \operatorname{round}(log(n))$, the kernel parameter of kernel function is set as $\varsigma =25$, and the termination condition of the s-iPDCA for solving RCKSDPP is set as $\frac{\|\mathbf{V}^{k+1}-\mathbf{U}^{k}\|_F}{1+\|\mathbf{U}^{k}\|_F}\leq 1 \times 10^{-4}$ while the termination precision of boundary points method for solving KSSDPP is set as $1 \times 10^{-5}$. The average recognition accuracy and the standard deviation across 50 runs of tests of KSSDPP, KPCA and RCKSDPP are shown in Table \ref{tab:2}.  From Table \ref{tab:2}, we can know that the RCKSDPP outperforms the KSSDPP and KPCA on the  ORL database for all situations. What's more, the larger the training set is, the higher the recognition accuracy was obtained for each model.\par
\begin{table}
\caption{Dimension reduction results: recognition accuracy  $\left(\text{Re}\pm\text{std}\%\right)$, average optimal value $(J)$ and average solving time $(t/s)$  on the ORL database}
\label{tab:2}       
\begin{tabular}{cccccccccccccc}
\hline\\[2pt]
Partitions&\multicolumn{4}{c}{ KSSDPP}&\multicolumn{4}{c}{RCKSDPP}&\multicolumn{2}{c}{KPCA}\\[2pt]
    \cmidrule(lr){2-5} \cmidrule(lr){6-9} \cmidrule(lr){10-11}\\[2pt]
    &$\text{Re}\pm\text{std}$&$J$&Iter&t/s&$\text{Re}\pm\text{std}$&$J$&Iter&t/s&$\text{Re}\pm\text{std}$&t/s\\[5pt]
\hline\\[2pt]
$G_2/ T_8$ & $80.15\pm3.26$&6.53e-2&1523&19.47& $83.22\pm2.57$&2.22e-4&654&8.11&$79.75\pm2.41$ &0.95\\[10pt] 
$G_3/ T_7$ & $88.11\pm2.56$&9.88e-2&1242&21.67& $90.90\pm 2.16$&1.27e-3&682&17.04&$84.67\pm12.08$ &0.98 \\[10pt]
$G_4/ T_6$ & $92.32\pm1.80$&1.33e-1&1366&38.33 & $94.47\pm 1.70$&2.25e-3&747&22.31&$88.59\pm12.58$ &1.00 \\[10pt]
$G_5/T_5$ & $95.20\pm0.76$&1.245e-1&1840&72.43 & $96.46 \pm 1.56$&4.6e-3&716&43.42 &$93.39\pm2.16$ &1.01\\[10pt]
$G_6/T_4$ & $96.37\pm1.49$&1.47e-1&1751&92.93 & $97.15 \pm  1.66$&5.93e-3&684&50.55&$95.01\pm1.63$ &1.03 \\[10pt]
$G_7/ T_3$ & $97.00\pm0.75$&2.14e-1&1710&158.48 & $98.25 \pm 1.19$&1.15e-2&634&67.38&$96.53\pm1.64$ &1.04 \\[10pt]
$G_8/T_2$ & $98.02\pm1.62$&2.11e-1&1276&160.97 & $98.70 \pm 1.58$&1.41e-2&641 &72.55 &$97.15\pm2.02$ &1.05\\[5pt]
\noalign{\smallskip}\hline
\end{tabular}
\end{table}

Compared to the ORL database, the YaleB  database has different illuminations, which makes it diffcult to perform recognition task on the YaleB database. In the DR experiment on the YaleB, We set $r = 45$, set the neighborhood size of k-nn for the RCKSDPP  as $k = 3$, and in each run of RCKSDPP, KSSDPP and KPDCA,  20 individuals are randomly selected for face recognition. The different suitable kernel parameter is chosen for each of these three models, e.g., $\varsigma =7$ for the KSSDPP and RCKSDPP, while $\varsigma =2000$ for the KPCA. The termination condition of the s-iPDCA for solving RCKSDPP is set as $\frac{\|\mathbf{V}^{k+1}-\mathbf{U}^{k}\|_F}{1+\|\mathbf{U}^{k}\|_F}\leq 7 \times 10^{-5}$ and the termination precision of boundary points method for solving KSSDP is set as $1 \times 10^{-5}$. The comparison results on this database is shown in Table \ref{tab:3}, which shows that the unsupervised method, the KPCA, has much lower recognition accuracy  than the RCKSDPP and KSSDPP on the YaleB face database.  What's more, the RCKSDPP outperforms the KSSDPP, which means that RCKSDPP is more powerful to reduce the dimension of complex data than the KSSDPP.\par
\begin{table}
\caption{Dimension reduction results: recognition accuracy  $\left(\text{Re}\pm\text{std}\%\right)$, average optimal value $(J)$ and average solving time $(t/s)$  on the YaleB database}
\label{tab:3}       
\begin{tabular}{ccccccccccc}
\hline\\[2pt] 
Partitions&\multicolumn{4}{c}{KSSDPP}&\multicolumn{4}{c}{RCKSDPP}&\multicolumn{2}{c}{KPCA}\\[2pt]
    \cmidrule(lr){2-5} \cmidrule(lr){6-9} \cmidrule(lr){10-11} \\[2pt]
    &$\text{Re}\pm\text{std}$&$J$&Iter&t/s&$\text{Re}\pm\text{std}$&$J$&Iter&t/s&$\text{Re}\pm\text{std}$&t/s\\[5pt]
\hline\\[2pt]
$G_{10}/ T_{50}$& $68.24\pm 1.67$&7.13e-2&1918&42.05& $76.28 \pm 3.17$ &3.03e-3&718&40.70&$55.43 \pm 2.57$ &1.52\\[10pt] 
$G_{20}/T_{40}$& $79.98\pm 2.53$&1.30e-1&2093&169.73 & $89.03\pm 1.53$&1.64e-2&575&86.21 &$67.13 \pm 2.63$ &1.81\\[10pt] 
$G_{30}/T_{30}$& $85.70\pm 1.76$&1.52e-1&2212&379.76 & $93.67 \pm  1.14$&2.59e-2&531&150.58&$71.35 \pm 2.65$ &2.15 \\[10pt] 
$G_{40}/ T_{20}$& $87.10\pm 1.76$&1.50e-1&2389&702.58 & $94.00 \pm 1.77$&3.28e-2&459&193.90&$74.33 \pm 3.00$ &2.46 \\[10pt] 
$G_{50}/T_{10}$&$87.50\pm 2.89$ &1.91e-1&2476&1494.51 & $95.05 \pm 1.45$&4.43e-2&426&298.95&$74.57 \pm 3.74$ &2.80\\[5pt] 
\noalign{\smallskip}\hline
\end{tabular}
\end{table}
From Table \ref{tab:2} and Table \ref{tab:3}, we know that the RCKSDPP always outperforms the KSSDPP from computation times and recognition accuracy. This results demonstrate that the RCKSDPP is effective to reduce the dimension of image data. What's more, the larger the training set size is, the higher the recognition accuracy of KSSDPP, RCKSDPP and KPCA are obtained.
\section{Conclusion}\label{sec:6}
 In this paper, the supervised distance preserving projections (SDPP) for dimension reduction has been considered, which is reformulated into the rank constrained least squares semidefinite programming (RCLSSDP). To address the difficulty bringing by the rank constraint, we introduce the DC regularization strategy to transfer the RCLSSDP into LSSDP with DC regularization.  Under the framework of DC approach, we propose an efficient algorithm, the  inexact proximal DC algorithm with sieving strategy (s-iPDCA), for solving the DCLSSDP. Then, we  show that the sequence generated by s-iPDCA globally converges to the stationary points of corresponding DC probelm. In addition, based on the dual of the strongly convex subprobelm of s-iPDCA, an efficient accelerated block coordinate descent (ABCD) method  is designed.  What's more, an efficient inexact strategy is  employed to solve the subprobelm of s-iPDCA.
 Moreover, the low rank structure of solution is considered to reduce the storage cost and computation cost. Finally, we compare our s-iPDCA with the classical PDCA and the PDCA with extrapolation (PDCAe) for solving the  RCLSSDP by performing DR on the COIL-20 database, the results show that the s-iPDCA is  more efficient to solve the RCLSSDP. We also apply the rank constrained kernel  SDPP (RCKSDPP) to perform DR  for face recognition on the ORL and the YaleB database, the resluts demonstrate that the  RCKSDPP outperforms the Kernal SSDPP (KSSDPP) and the kernel principal component analysis (KPCA) on recognition accuracy.

%
%

\bibliographystyle{spmpsci_unsrt}    
\bibliography{ref}   



\end{document}